\documentclass[a4paper,11pt]{article}

\usepackage{geometry}     
\usepackage{amsmath}
\usepackage{graphicx}
\usepackage{amssymb}
\usepackage{epstopdf}
\usepackage{float}
\usepackage{caption}
\usepackage{hyperref}
\usepackage{mathtools}
\usepackage{color}
\usepackage{listings}
\usepackage{floatpag}
\usepackage[normalem]{ulem}
\usepackage{mathrsfs}
\usepackage{bbm}
\usepackage[dvipsnames]{xcolor}
\usepackage{todonotes}
\usepackage{enumitem}
\usepackage{amsthm}
\usepackage{nicefrac}
\usepackage{thmtools}
\usepackage{thm-restate}
\usepackage[nottoc,notlof]{tocbibind}

\newtheorem{theorem}{Theorem}[section]

\newtheorem{proposition}[theorem]{Proposition}

\newtheorem{remark}[theorem]{Remark}
\newtheorem{lemma}[theorem]{Lemma}
\newtheorem{assumption}[theorem]{Assumption}
\newtheorem*{theorem*}{Theorem}

\theoremstyle{definition}
\newtheorem{definition}[theorem]{Definition}

\usepackage{wasysym}
\usepackage{textcomp}

\usepackage[style=alphabetic,isbn=false,doi=false,maxnames=10,backend=biber]{biblatex}
\renewbibmacro{in:}{}
\DeclareFieldFormat
  [article,inbook,incollection,inproceedings,patent,thesis,unpublished]
  {title}{#1\isdot}

\newcommand{\cL}{\mathcal{L}}

\newcommand{\de}{\operatorname{d}}

\newcommand{\sC}{\mathscr{C}}
\newcommand{\sG}{\mathscr{G}}

\newcommand{\sN}{\mathscr{N}}

\DeclareMathOperator{\E}{\mathbb{E}}
\DeclareMathOperator{\p}{\mathbb{P}}

\newcommand{\pl}{_{\mathbb{P}}^{\log N}}

\usepackage{parskip}
\usepackage{tcolorbox}
\usepackage{framed}
\usepackage{float}
\usepackage[caption = false]{subfig}
\usepackage{graphicx}
\usepackage{makecell}
\usepackage{appendix}
\usepackage{nicefrac}

\usepackage{soul}
\newcommand{\myul}[2][black]{\setulcolor{#1}\ul{#2}\setulcolor{black}}
\def\mathunderline#1#2{\color{#1}\underline{{\color{black}#2}}\color{black}}

\title{Discursive Voter Models on the Supercritical Scale-Free Network}
\author{John Fernley}

\begin{document}

\begin{center}
{\LARGE  Discursive Voter Models on the\\
\vspace{0.5em}
 Supercritical Scale-Free Network}\\
\vspace{1.5em}
{John Fernley\footnote{HUN-REN Alfr\'ed R\'enyi Institute of Mathematics,
			Re\'altanoda utca 13-15,
            Budapest,
            1053, 
            Hungary, {\tt fernley@renyi.hu}.}}
\end{center}

\vspace{2em}

\begin{abstract}
The voter model is a classical interacting particle system, modelling how global consensus is formed by local imitation. 
We  analyse the time to consensus for a particular family of voter models when the underlying structure is a scale-free inhomogeneous random graph, in the high edge density regime where this graph features a giant component. In this regime, we verify that the polynomial orders of consensus agree with those of their mean-field approximation in \cite{moinet2018generalized}.

This ``discursive'' family of models has a symmetrised interaction to better model discussions, and is indexed by a temperature parameter which, for certain parameters of the power law tail of the network's degree distribution, is seen to produce two distinct phases of consensus speed. Our proofs rely on the well-known duality to coalescing random walks and a novel bound on the mixing time of these walks, using the known fast mixing of the Erd\H{o}s-R\'enyi giant subgraph. Unlike in the subcritical case \cite{fernley2019voter} which requires tail exponent of the limiting degree distribution $\tau=1+\nicefrac{1}{\gamma}>3$ as well as low edge density, in the giant component case we also address the ``ultrasmall world'' power law exponents $\tau \in (2,3]$.

{\scriptsize

	}
\end{abstract}

\section{Introduction}

The most studied voter model to be put on a graph involves \emph{pulling} opinions from a uniform random neighbour and so is dual to the simple random walk. However, any irreducible Markov dynamic can reasonably replace this simple random walk, and in fact the first work to put the voter model on a graph was with a symmetrised dynamic \cite{clifford1973model}.
These symmetrised \emph{push-pull} dynamics are better for modelling discussions, in that either side can be swayed, and so we call it the discursive voter model. This model was recently also considered by \cite{moinet2018generalized, sood2008voter} in the physics literature, and something similarly symmetrical also in the ``oblivious'' voter model of \cite{cooper2016discordant}.

On a scale-free network with $N$ vertices these models are very influenced by the presence of degrees polynomially large in $N$, and so we also introduce a ``temperature'' parameter to control the influence of these large polynomials -- with any $\theta \in \mathbb{R}$, a vertex which was interacting at rate $\de(v)$ in the version of \cite{sood2008voter} would instead interact at rate $\de(v)^\theta$. $\theta=0$, then, and also $\theta=1$, are the most studied models in this family. 

\begin{definition}[Discursive voter model]
Fix $\theta \in \mathbb{R}$ and a graph on vertices $[N]=\{1,\dots,N\}$. 
Given $\eta \in \{0,1\}^N$ and $i\neq j \in [N]$, define
\[ \eta^{i \leftarrow j} (k) = \left\{ \begin{array}{ll} \eta(j)  & \mbox{if } k = i, \\ \eta(k)  & \mbox{if } k \in [N] \setminus \{ i \} . \end{array}\right. \]
The discursive voter model $(\eta_t)_{t \geq 0}$ with temperature parameter $\theta$ is then the Markov process with state space $\{0,1\}^N$ and generator $\cL$ defined by
\[
\cL f(\eta)=
\sum_{i=1}^N 
\de(i)^{\theta}
\sum_{j \sim i}^N
\frac{1}{\de(i)}
\left(
\frac{1}{2} f\left(\eta^{j \leftarrow i}\right)
+
\frac{1}{2} f\left(\eta^{i \leftarrow j}\right)
-f\left(\eta\right)
\right).
\]
\end{definition}

This expression should be interpreted as each site $i$ starting a discussion at rate $\de(i)^{\theta}$: this involves picking a uniform neighbour, and then resolving to consensus of the two parties by picking a a uniform opinion from their two opinions. Thus opinion moves
\begin{equation}\label{eq_dual_dynamic}
i \rightarrow j \text{\quad at rate \quad}
\frac{\de(i)^{\theta -1}+\de(j)^{\theta -1}}{2}\mathbbm{1}_{i \sim j}
\end{equation}
which are also the dynamics of the dual coalescing walkers. This duality is by time reversal, see \cite[Section 4]{fernley2019voter} for a detailed explanation. The most important consequence is that consensus time $\tau_{\rm cons}^{[N]}$ starting with $N$ distinct opinions can be coupled with the coalescence time $\tau_{\rm coal}^{\mathbbm{1}_{[N]}}$ of a full occupation of walkers with the dual dynamic such that almost surely 
$
\tau_{\rm coal}^{\mathbbm{1}_{[N]}}=\tau_{\rm cons}^{[N]}.
$
 Our interest is primarily in the case of two distinct opinions, but we can trivially dominate this consensus time by the time of consensus from $N$ opinions.

Because the rates \eqref{eq_dual_dynamic} are symmetric in $i$ and $j$, their invariant distribution (on each component) is uniform and so the number of vertices with opinion $1$ over the whole network is a martingale \cite[Proposition 3.1]{cox2016convergence}.

For a rank one scale-free network, there are many similar models which are frequently seen to have similar properties as environments. 

\begin{definition}[Simplified Norros-Reittu graph $G_N$]\label{def_SNR}
The Simplified Norros-Reittu (SNR) graph, denoted $G_N$ and with parameters $\beta>0$, $\gamma \in [0,1)$, is the simple graph with vertex set $[N]=\{1,\dots,N\}$ and each edge $\{i,j\}$ independently present with probability
\[
p_{ij}=1-\exp \left( -\beta N^{2\gamma-1}i^{-\gamma}j^{-\gamma} \right)
\]
for every pair of distinct vertices $i,j \in [N]$.
\end{definition}

When $\gamma\neq 0$ we can translate to parameter $\tau=1+\nicefrac{1}{\gamma}$ and see that in $G_N$
\[
\mathbb{P}\left( \de(U_N)=k \right) \sim k^{-\tau},
\] 
where $U_N$ denotes a uniform random variable in $[N]$ 
and $k=k(N) \rightarrow \infty$ arbitrarily slowly in $N$. Thus $\tau$ is said to be the power law tail of the degree distribution.

Definition \ref{def_SNR} is equivalent to a wide range of more natural rank one models when $\gamma<\nicefrac{1}{2}$, because $1-e^{-x}\sim x$ and by applying \cite[Theorem 6.18]{van2016random}. See also \cite[Definition 2.2]{fernley2019voter} for the equivalence class. This definition from that class is particularly nice mathematically, though, as it is the simplified (``flattened'') version of the natural Norros-Reittu multigraph of Definition \ref{def_mnr}.

Our main theorems use the notation $\Theta_{\mathbb{P}}^{\log N}(\cdot)$ to denote a two-sided order bound satisfied with high probability and allowing a poly-logarithmic correction factor. That is,
\[
f(N)=\Theta_{\mathbb{P}}^{\log N}(g(N))
\iff
 \exists C >0 : \mathbb{P}\left(
 \frac{g(N)}{\log^{2C} N}
 <
\frac{f(N)}{\log^C N}
<
g(N)
 \right)\rightarrow 1.
\]
We will later use general Landau notation with this superscript or subscript to denote a polylogarithmic correction or bound holding with high probability, respectively.

When $\beta+2\gamma<1$, the largest component in the network $G_N$ has $\Theta_{\mathbb{P}}^{\log N}(N^\gamma)$ vertices -- in particular, there is no giant component. We define the consensus time for the voter model on a disconnected graph as the first hitting time of an absorbing state. 

That is, if $C_1, \ldots, C_\ell \subset [N]$ are the components of $G_N$, the \emph{consensus time} is
\[ \tau_{\rm cons} = \inf\{ t \geq 0 \, : \, \eta_t|_{C_i} \mbox{ is constant for each } i \in [\ell] \} \]
and in the subcritical case we then have the following consensus orders.

\begin{theorem}[ {\cite[Theorem 2.6]{fernley2019voter}} ]\label{thm_subcrit}
Take $\gamma\in[0,\tfrac{1}{2})$ and $0<\beta<1-2\gamma$. 
Then for the discursive voter model on $G_N$ from initial opinions distributed as $\mu_u$ of the Bernoulli process (see Definition \ref{def_bernoulli}) with $u \in (0,1)$, we have
\[
\mathbf{E}^{\theta}_{\mu_u}(\tau_{\text{\textnormal{cons}}}|G_N)=
\Theta^{\log N}_{\mathbb{P}}\left( N^{c} \right)
\text{
where } 
c=
\begin{cases}
 \frac{\gamma}{2-2\gamma}  & \theta \geq \frac{3-4\gamma}{2-2\gamma} , \\
 \gamma(2-\theta) & 1 < \theta < \frac{3-4\gamma}{2-2\gamma} , \\
 \gamma & 2\gamma \leq \theta \leq 1 ,\\
 \frac{\gamma(2-\theta)}{2-2\gamma} & \theta < 2\gamma.
\end{cases}
\]
\end{theorem}

This result is explained by components of order $\Theta_{\mathbb{P}}^{\log N}(N^\gamma)$ or $\Theta_{\mathbb{P}}^{\log N}(N^{\nicefrac{\gamma}{2-2\gamma}})$ in the different regimes above. 
Note that the subcritical Erd\H{o}s-R\'enyi case, somewhat special due to light tail of the empirical degree distribution and included in Theorem \ref{thm_subcrit} as parameter region $\{\beta<1, \gamma=0\}$, can be seen to have $\Theta_{\p}(\log^3 N)$ consensus times for any $\theta$. 

Instead, in the giant component regime, all components apart from the giant have \emph{maximally} size $O_{\mathbb{P}}^{\log N}(1)$. 
Hence, in this article we will find orders for $\tau_{\text{\textnormal{cons}}}$ which are always produced by the consensus time of a unique giant component of order $\Theta_{\mathbb{P}}^{\log N}(N)$. 

\subsection{Main results}

Of our three main theorems, the first covers the ``small world'' parameters $\gamma < \nicefrac{1}{2}$, including the Erd\H{o}s-R\'enyi graph $\gamma=0$, where typical distances in the network grow logarithmically. In this $\gamma < \nicefrac{1}{2}$ case, the network $G_N$ can in fact be seen by \cite[Theorem 6.18]{van2016random} to be equivalent to a range of equivalent network definitions. In particular, the theorem also applies to the Chung-Lu network with edge probabilities $p_{ij}=1 \wedge\left( \beta N^{2\gamma-1}i^{-\gamma}j^{-\gamma} \right)$.

\begin{definition}[$\mu_u$]\label{def_bernoulli}
The Bernoulli process in $\{0,1\}^{[N]}$ with parameter $u \in [0,1]$ has measure $\mu_u$ such that at each $v \in [N]$ is an independent $\operatorname{Bernoulli}(u)$ random variable.
\end{definition}

\begin{theorem}\label{theorem_disc_cons_small}
Take $\gamma\in \left[0,\tfrac{1}{2}\right]$ and $\beta >1-2\gamma$. Then for the discursive voter model on $G_N$ from initial opinions distributed as $\mu_u$ of the Bernoulli process with $u \in (0,1)$, we have
\[
\mathbb{E}^{\theta}_{\mu_u}(\tau_{\text{\textnormal{cons}}}|G_N)=
\Theta_{\mathbb{P}}^{\log N}\left( N \right)
\]
whenever $\theta \in (-\infty,2]$.
\end{theorem}

Note that the condition $\beta >1-2\gamma$ is the full supercritical regime in the sense that it gives a $\Theta_{\mathbb{P}}(N)$ component, and the largest component is $o_{\mathbb{P}}(N)$ whenever this condition fails. Apart from the critical line $\{\beta+2\gamma=1\}$, this result completes the picture of Theorem \ref{thm_subcrit} over all positive $\beta$ and  $\theta \in (-\infty,2]$.

Our second theorem is instead on the ``ultrasmall world'' parameters $\gamma > \nicefrac{1}{2}$, where typical distances are expected to be doubly logarithmic, and here we see a second phase of superlinear consensus.

\begin{theorem}\label{theorem_disc_cons_ultrasmall}
Take $\gamma\in \left(\tfrac{1}{2},1\right)$ and $\beta > 2\log 2$. Then for the discursive voter model on $G_N$ from initial opinions distributed as $\mu_u$ of the Bernoulli process with $u \in (0,1)$, we have
\[
\mathbb{E}^{\theta}_{\mu_u}(\tau_{\text{\textnormal{cons}}}|G_N)=
\begin{cases}
\Theta_{\mathbb{P}}^{\log N}\left( N \right), &\theta \in \left(-\infty, \frac{1}{\gamma}\right],\\
\Theta_{\mathbb{P}}^{\log N}\left( N^{2-\gamma \theta}\right), &\theta \in \left(\frac{1}{\gamma}, 2\right].
\end{cases}
\]
\end{theorem}

Here $2\log 2 \approx 1.386$. We can easily provide a heuristic for these orders of coalescence time when all the models have the dual dynamics \eqref{eq_dual_dynamic}, as then from the uniform stationary distribution it is natural to conjecture that two walkers on the giant component meet in $\Theta(N)$ total steps. To approximate the rate of stepping, we simply observe that the ergodic rate is on the order $\nicefrac{\sum \de^\theta}{N}$, again using the uniform stationary distribution. 
The mean-field interactions (where $i$ chooses to interact with any vertex $j \in [N]$ with probability proportional to $j^{-\gamma}$) were considered by \cite{moinet2018generalized} who give the consensus order 
$\nicefrac{N^2}{\sum \de^\theta}$
 which agrees with the polynomial orders in Theorems \ref{theorem_disc_cons_small} and \ref{theorem_disc_cons_ultrasmall}. So, we verify here that the supercritical phase of a rank-one network $G_N$ (rank-one in the multigraph sense of Definition \ref{def_gsnr}) is indeed well approximated by the mean field environment, when it comes to voter model consensus speed.

\begin{remark}[The polylogarithmic factor]
Despite this simple explanation, it is hard to upper bound meeting on an order strictly faster than $N$ as we cannot wait for an $\Omega(1)$ mixing period in between checking for a meeting of two uniform walkers (with probability $O(\nicefrac{1}{N})$). Instead, we get the upper bound by controlling the amount of time walkers spend in the neighbourhood of the highest degree vertices and then use a partially observed \cite[Section 2.7.1]{aldous-fill-2014} version of the chain to find a meeting time.

Results \cite[Theorem 1.3]{oliveira2013mean} and \cite[Proposition 2.5]{cox2016convergence} describe the mean field for the voter model on condition $t_{\rm mix} \ll t^\pi_{\rm meet}$, which we verify in this article when $\theta \geq 1$, with the precise limit shape of Kingman’s coalescent
\[
\sum_{i=2}^\infty \nicefrac{E_i}{\binom{i}{2}},
\quad
\text{where }
E_i \stackrel{\rm i.i.d.}{\sim} \operatorname{Exp}(1).
\]
However, this shape is observed on the timescale $t^\pi_{\rm meet}$, the expected meeting time of two stationary walkers, which still with our methods we can only control up to a polylogarithmic factor.

We follow \cite{durrett2010some} in conjecturing via Aldous' ``Poisson Clumping Heuristic'' \cite{aldous2013probability} that the logarithmic corrections are only an artifact of the proof and the order of the mean consensus time is the exact polynomial order given without any polylogarithmic correction. That is, the same order as the result of \cite{moinet2018generalized} and tight to the lower bound of Proposition \ref{lower_meeting_bound}. 
This heuristic uses Kac's formula \cite[(2.24)]{aldous-fill-2014} for the return time to the ``diagonal'' set $\{(i,i):i \in \sC_{\rm max}\}$. Applied to any voter model with dual chain having stationary distribution $\pi$ and vertex rates $q$, it tells us that for the distribution $\rho \propto q$
\[
t^\rho_{\rm meet}:=
\mathbb{E}_{\rho \otimes \rho} \left( t_{\rm meet} \right)
=
\frac{1-\sum_v \pi^2(v)}{2\sum_v \pi^2(v) q(v)}
\]
which agrees with the polynomial orders of our main theorems. Proving such a result without polylogarithmic corrections would require a more detailed structural understanding of the network to establish accurately the probability of mixing without meeting and so relate $t^\pi_{\rm meet}$ to $t^\rho_{\rm meet}$, for example using \cite[Equation 3.18]{cox2016convergence}.
\end{remark}

The model continues to make sense when $\theta >2$, but note that these are the parameter values where consensus time on a large star graph tends to zero polynomially fast and so not suited to the Chernoff bounds of \cite{lezaud_chernoff} which we otherwise rely on. Scale-free network dynamics are commonly, and certainly for this model, driven by the vertices of largest degree and so Theorems \ref{theorem_disc_cons_small} and \ref{theorem_disc_cons_ultrasmall} cover all parameters $\theta \in \mathbb{R}$ not driven by this degenerate behaviour.

\begin{definition}[VSRW]\label{def_vsrw}
The variable speed random walk on a graph $G=(V,E)$ is defined by generator matrix
\[
\forall i \neq j \quad
Q(i,j)=
\mathbbm{1}_{i \sim j}
\]
\end{definition}

In the ultrasmall case, we need to bound the mixing time (with the standard definition total variation threshold $\nicefrac{1}{e}$, see \eqref{eq_mixing_def} for the precise definition) and so need large enough $\beta$ to have room for the mixing construction in the proof of Theorem \ref{theorem_erw_mixing}. This loose mixing bound gives the polylogarithmic correction accuracy sufficient for Theorem \ref{theorem_disc_cons_ultrasmall}.

\begin{restatable}{theorem}{erwmixing}
\label{theorem_erw_mixing}
Given $\beta \geq 3$, the SNR network of Definition \ref{def_SNR} has VSRW mixing time
\[
t_{\rm mix}=O^{\log N}_{\mathbb{P}}\left( 1 \right).
\]
\end{restatable}

Note that this theorem applies to the full range of power laws $\gamma \in [0,1)$. There are many existing mixing results for scale-free networks which are configuration models (e.g. \cite{zbMATH06093928} who impose a minimum degree of $3$), but the rank one SNR in the ultrasmall regime is not a configuration model.

To obtain this mixing bound, we use that $\beta>1$ to find an Erd\H{o}s-R\'enyi giant component as a subgraph of the giant component of $G_N$. This is known to be fast mixing, and so by growing it to a spanning subgraph of the complete giant by attaching trees of maximal size $O^{\log N}_{\mathbb{P}}\left( 1 \right)$ we produce a structure which can also be seen to also have fast mixing. Once we obtain a fast-mixing spanning subgraph, the symmetric generator is used in applying \cite[Corollary 3.28]{aldous-fill-2014} to argue that completing internal edges of this subgraph can only further accelerate mixing and so we obtain a mixing bound on the \emph{induced} subgraph which is simply the complete giant component.

Obtaining a mixing bound for the VSRW achieves a lot in this work, as the dual walkers \eqref{eq_dual_dynamic} are a family of symmetric dynamics with rates non-decreasing in $\theta$. Hence, by again \cite[Corollary 3.28]{aldous-fill-2014}, the bound of Theorem \ref{theorem_erw_mixing} also applies to every dynamic with $\theta \geq 1$, which are the hard cases of Theorem \ref{theorem_disc_cons_ultrasmall}.

\begin{assumption}\label{assumption_large_beta}
While we state and prove Theorem \ref{theorem_erw_mixing} with simpler condition $\beta \geq 3$, our methods actually apply to any $\beta, \gamma$ satisfying 
\[
\rho >\frac{1}{2}
\text{ and }
\int_0^1
x^{-2\gamma}
(1-\rho)^{\left( \frac{1}{1-\gamma}-1+\gamma \right) x^{-\gamma}}
{\rm d} x<1
\]
for the unique $\rho = \rho(\beta) \in (0,1)$ satisfying $1-\rho=e^{-\beta \rho}$.
\end{assumption}

\begin{remark}[Smaller supercritical $\beta$]
We use the Erd\H{o}s-R\'enyi subgraph in the SNR network to control mixing, and so any modification of our approach would still require at least $ \beta > 1 $ while the giant component regime for this network is really $\{ \beta + 2 \gamma > 1 \}$.

When $\gamma \leq \nicefrac{1}{\theta}$ our results on the consensus time do not need mixing control and in fact apply to all $\{\beta>(1-2\gamma)\vee 0\}$, and at larger $\gamma$ we require the technical assumption \ref{assumption_large_beta}. These larger $\gamma$ are at least $\gamma\geq\nicefrac{1}{2}$ and we show in Proposition \ref{prop_beta_2log} that there the integral condition doesn't play a role. Hence we can state Theorem \ref{theorem_disc_cons_ultrasmall} with just condition $\rho>\nicefrac{1}{2}$ or equivalently $\beta>2\log 2$. 
To control mixing when $\gamma\leq \nicefrac{1}{2}$, however, we do need the integral condition 
and in Proposition \ref{prop_beta_3} we show that for that $\beta>2.17$ is sufficient. 

To show our results for the parameter values with $\beta+2\gamma>1$ but without the conditions of Assumption \ref{assumption_large_beta}, we would need a bound on the mixing time for the VSRW on a graph with no Erd\H{o}s-R\'enyi subgraph.
 Still, it is widely believed that the network would continue to be fast-mixing through the whole supercritical region, and therefore we can fairly confidently conjecture that the exponents we prove for the region $\{ \beta \geq 2\log 2 \}$ wouldn't change for all smaller $\beta$ in the region $\{\gamma>\nicefrac{1}{2}\}$.
\end{remark}

\begin{figure}[h]
\centering
\includegraphics[scale=0.9]{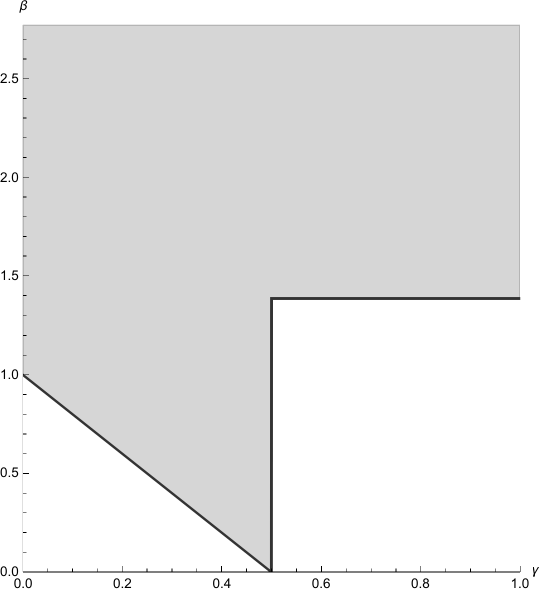}
\caption{
We prove Theorems \ref{theorem_disc_cons_small} and \ref{theorem_disc_cons_ultrasmall} in the highlighted region (when $\gamma\geq\nicefrac{1}{2}$ requiring Assumption \ref{assumption_large_beta}, the integral condition of which we cannot see in this plot), and conjecture the same exponents would be seen where $\{\gamma>\nicefrac{1}{2},\beta \leq 2\log 2\}$. In the other hole $\{\beta+2\gamma<1\}$, we see a very different subcritical graph structure and so the different consensus time exponents of Theorem \ref{thm_subcrit}. }\label{fig:parameter_regions}
\end{figure}

\section{Paths in the Network}

While it is perfectly possible to define Poissonian inhomogenous random graphs that are not rank one, we will mostly be interested in \emph{rank one} Norros-Reittu graphs as in the original definition of \cite{norros2006conditionally}.

\begin{definition}[General Multigraph Norros-Reittu]\label{def_mnr}
We parametrise Multigraph Norros-Reittu (MNR) by a vertex weight profile
\[
\operatorname{f}: (0,1] \rightarrow (0,\infty),
\]
and then the multigraph has independently
\[
\operatorname{Pois}\left( \frac{1}{N} \operatorname{f}\left( \frac{i}{N} \right) \operatorname{f}\left( \frac{j}{N} \right)\right)
\]
edges between each pair $i,j \in [N]$, including pairs with $i=j$. 
\end{definition}

The Simplified Norros-Reittu (SNR) model is obtained from the multigraph model by \emph{flattening}, i.e. removing loops (edges only incident to one vertex), and then reducing all edge counts to $1$ if they are greater.

\begin{definition}[General Simplified Norros-Reittu]\label{def_gsnr}
Given the same $\operatorname{f}$ of Definition \ref{def_mnr}, we constuct the Simplified Norros-Reittu network by inserting an edge independently between each $i \in [N]$ and $j \in [N]\setminus\{i\}$ with probability
\[
1-\exp\left( - \frac{1}{N} \operatorname{f}\left( \frac{i}{N} \right) \operatorname{f}\left( \frac{j}{N} \right)\right).
\]
\end{definition}

For example, the kernel defining our main network of interest in Definition \ref{def_SNR} is
\[
\operatorname{f}(x)=\sqrt{\beta}x^{-\gamma}.
\]

When $\gamma<\nicefrac{1}{2}$ we can bound the network diameter by coupling to a uniform random graph and using results in \cite{fernholz2007diameter}. However, we also require a bound for $\gamma\geq\nicefrac{1}{2}$ which, to our knowledge, does not exist in the literature.

Rather than redevelop the theory, we will obtain this bound quicker by repeatedly applying the super- and sub-critical diameter theorems in \cite{bollobas2007phase}.

\begin{theorem}\label{my_diameter_bound}
The rank one NR network with any weight function which is supercritical
\[\int_0^1 \operatorname{f}^2 \in (1,\infty ] \]
(see Proposition \ref{prop_rankone_snr_giant}) and bounded below
\[
\min \operatorname{f} >0
\]
has componentwise diameter $O_{\mathbb{P}}(\log N)$.
\end{theorem}

\begin{proof}
We can assume that $f$ is a nonincreasing function without loss of generality because the purpose of the function $f$ is to define a weight measure on $(0,\infty)$ and so can always obtain a nonincreasing equivalent by ``reordering'', replacing $f$ with the weight function 
$
x
\mapsto
\min \left\{
C>0
:
\left|\{
f>C
\}\right|\leq x
\right\}
$.

Then the core of the argument is that we will lower bound the network model with a supercritical network that has finitely many types.
Define for $M \in \mathbb{N}^+$
\[
\operatorname{f}_M(x):=\operatorname{f}\left( \frac{\lceil Mx \rceil}{M} \right)
\]
so that we have stochastic domination of the edges in the rank one networks written $\mathcal{N}\operatorname{f}_M \preceq \mathcal{N}\operatorname{f}$,
where $\mathcal{N}g$ denotes the NR rank one network with weight function $g$ determining the edge means before flattening.

Note also that because $\operatorname{f}_M$ is bounded, from \cite[Theorem 6.18]{van2016random} we know that the simplified Norros-Reittu $\mathcal{N}\operatorname{f}_M$ and the Chung-Lu graph with kernel $\operatorname{f}_M$ are \emph{asymptotically equivalent}, that is that we can couple the entire graphs with high probability.

Hence we consider the norm of the Chung-Lu kernel which is given by \cite[Equation 16.8]{bollobas2007phase}.
As $M \rightarrow \infty$, using that $f$ in nonincreasing,
\[
\int_{M^{-1}}^1 \operatorname{f}_M^2= \frac{1}{M} \sum_{k=2}^M \operatorname{f}\left( \frac{k}{M} \right)^2 \rightarrow \int_0^1 \operatorname{f}^2 >1
\]
so that for $M$ large enough the \emph{truncated} model is also a supercritical rank one network, with bounded expected degree and finitely many types. In fact, if we take
\[
0<\epsilon < 1-\frac{1}{\int_0^1 \operatorname{f}^2}
\]
then $\sqrt{1-\epsilon}\operatorname{f}_M$ is still asymptotically supercritical. Colour the mass of the measure with density $\operatorname{f}_M$ such that a measure with density $\sqrt{1-\epsilon} \operatorname{f}_M$ is blue and the difference measure is red:
\[
\operatorname{f}_M = 
\mathunderline{blue}{ \sqrt{1-\epsilon} \operatorname{f}_M } 
+ \mathunderline{red}{ (1-\sqrt{1-\epsilon}) \operatorname{f}_M}.
\]

\begin{figure}[htp]
\centering
\includegraphics[width=10cm]{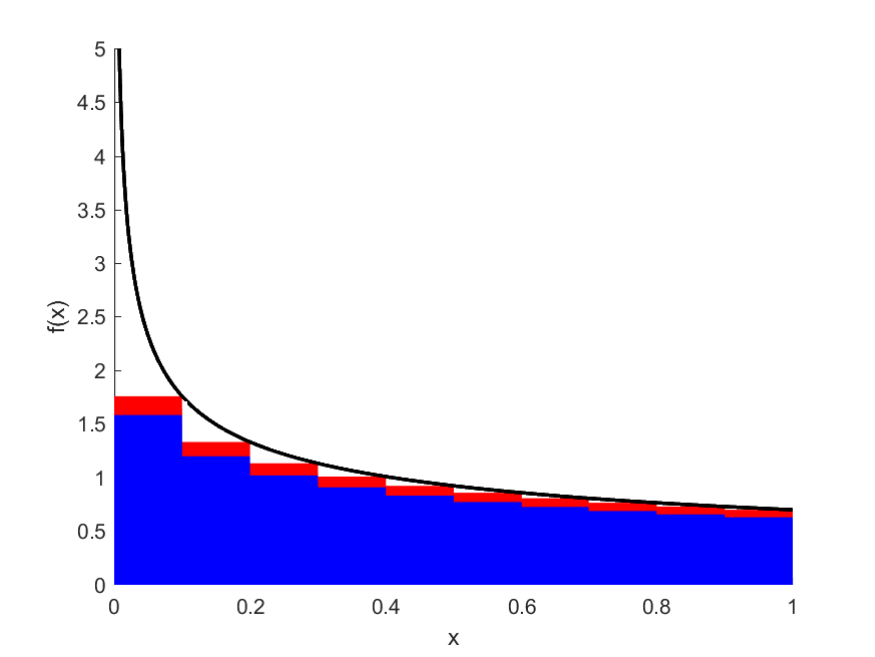}
\caption{We divide the degree mass and label by red or blue, for a typical rank one scale-free network with the step count $M=10$. Note there is some uncoloured mass.}\label{fig:area}
\end{figure}

The convenience of labelling mass by primary colours is that we can colour \emph{edges} \myul[blue]{blue} if they go from blue mass to blue mass and then \myul[magenta]{magenta} if they go from blue mass to red mass, et cetera. The idea behind such a picture is of a Norros-Reittu network where edges of differing colours arrive independently, and so we work with multigraph NR in the proof - before flattening, they are indeed independent.

Further all rank one kernels are trivially irreducible when restricted to the support of their weight function. We apply theorems in \cite{bollobas2007phase} (which are stated for the equivalent Chung-Lu version) to the network of only blue edges:

\begin{itemize}
\item By \cite[Theorem 3.1]{bollobas2007phase} (because the kernel is supercritical and irreducible) it has a giant component.
\item Because the network model has finitely many vertex types, by \cite[Theorem 3.16]{bollobas2007phase}, the componentwise diameter of this structure is $\Theta_{\mathbb{P}}(\log N)$.
\end{itemize}

So this subgraph of blue edges has a giant component whp, and it must be nested whp in the giant component of the edges of all colours.

After realising just the blue edges, we colour vertices in \myul{black} if they are in the largest component of the blue edge subgraph.
We now put aside the \myul[Plum]{dark magenta} edge mass, i.e. Poisson mass for \myul[magenta]{magenta} edges which would feature at least one black vertex.

Realise all the other Poisson edges in $\mathcal{N}\operatorname{f}_M$ from mass that is not \myul[blue]{blue} or \myul[Plum]{dark magenta}. These edges are either \myul[red]{red} or \myul[CarnationPink]{light magenta}. Thinking of this as the NR multigraph, these are all independent -- and questions of the diameter are the same on the multigraph or flattened version.

We finally realise the \myul[Plum]{dark magenta} edges, i.e. those \myul[magenta]{magenta} edges coming from a black vertex. Any vertex in $[N]$ and a particular black vertex are connected by a \myul[Plum]{dark magenta} edge with probability at least
\[
p_M:=
1-\exp\left(
-\frac{1}{N}(1-\sqrt{1-\epsilon}) \sqrt{1-\epsilon} \operatorname{f}(1)^2
\right)
=
\Omega\left( \frac{1}{N} \right)
\]
and so thus, after realising the number of black vertices at size $\Theta_{\mathbb{P}}(N)$, we can see vertices are incident to \myul[Plum]{dark magenta} edges with probability $\Theta_{\mathbb{P}}(1)$, and their incidence is entirely independent by the independence of edges for the multigraph NR network. 

Applying Lemma \ref{percolation_lemma}, we can say that any vertex on the graph is either on a component of small diameter or within maximal distance $O_{\mathbb{P}}(\log N)$ of a vertex of distance $1$ from the \myul[blue]{blue} giant. The \myul[blue]{blue} giant had logarithmic diameter and so any two connected vertices in $\mathcal{N}\operatorname{f}$ are of distance at most
\[
O_{\mathbb{P}}(\log N)+1+O_{\mathbb{P}}(\log N)+1+O_{\mathbb{P}}(\log N)=O_{\mathbb{P}}(\log N)
\]
by communicating through $\mathcal{N}\operatorname{f}_M$, and we have the result.
\end{proof}

\begin{lemma}\label{mypaths}
For the network $G_N$ of Definition \ref{def_SNR} with high probability, when $\beta>(1-2\gamma)\vee 0$, every pair of two vertices in the set
\[
S:=\left[ \left\lfloor
N
\log^{-\alpha}N
\right\rfloor \right],
\]
where $\alpha>\nicefrac{1}{\gamma}$, is simultaneously connected by a path of $O_{\mathbb{P}}(\log N)$ bounded weight vertices, in
\[
L_\epsilon:=[N] \setminus [\epsilon N],
\]
for some $\epsilon \in (0,1)$ sufficiently small.
\end{lemma}

\begin{proof}
The network $G_N$ has a Poisson number of edges between $i$ and $j$, with mean
\[
\beta N^{2\gamma-1}i^{-\gamma}j^{-\gamma}
= \frac{1}{N}
\left(\sqrt{\beta}\left(\frac{i}{N}\right)^{-\gamma}\right)
\left(\sqrt{\beta}\left(\frac{j}{N}\right)^{-\gamma}\right)
\]
and so corresponds precisely to the weight function
$
\operatorname{f} : x \mapsto
\sqrt{\beta} x^{-\gamma}.
$

The subnetwork induced by $L_\epsilon$ has rank one weight function $\operatorname{f} \mathbbm{1}_{( \epsilon,1]}$ and norm
\[
\int_\epsilon^1 \operatorname{f}^2>1
\]
given $\epsilon \in (0,1)$ sufficiently small.

Having fixed $\epsilon$ thus, Proposition \ref{prop_rankone_snr_giant} shows the induced rank one subnetwork on $L_\epsilon$ has a giant component $\mathscr{C}'$ which has $\mathscr{C}' \subseteq \mathscr{C}_{\text{max}}$ whp. Because $|\mathscr{C}'|=\Theta_{\mathbb{P}}(N)$ we have some $\delta>0$ such that the event 
$
\{|\mathscr{C}'|\geq \delta N\}
$ 
occurs with high probability.

Now take some $s \in S$. In the MNR version of the network conditioned on $
\{|\mathscr{C}'|\geq \delta N\}
$  we have $\operatorname{Pois}(\mu_s)$ edges between $s$ and $\mathscr{C}'$, where
\[
\mu_s \geq \frac{1}{N} f(1) f \left( \log^{-\alpha} N \right) \cdot \delta N 
= \beta \delta \log^{\alpha \gamma} N.
\]

Hence, the probability that there is no such edge is bounded by
\[
e^{-\beta \delta \log^{\alpha \gamma} N}=o\left(\frac{1}{N}\right)
\]
given $\alpha \gamma>1$. Then, using $|S|\leq N$, we can connect every $s \in S$ via the union bound. The result follows from Theorem \ref{my_diameter_bound}.

\end{proof}

\section{Markov Chains}

We have a few crucial results and definitions for general Markov chains which we will state in this section, before proving the mixing result of Theorem \ref{theorem_erw_mixing}.

$X = (X_t)_{t \geq 0}$ will be a reversible, irreducible Markov chain with state space $[N] = \{1, \ldots, N\}$, invariant measure $\pi$ and transition rates given by a generator matrix~$Q$. 
Because the chain is irreducible, the hitting time is defined as 
\[ t_{\rm hit} = \max_{k, j \in [N]} \E_k (T_j) \]
where $T_y := \inf\{ t \geq 0 \, : \, X_t = y \}$.

For our bound on the hitting time, we will make use of the well-known correspondence between 
Markov chains and electric networks, see e.g.~\cite{aldous-fill-2014, wilmer2009markov}.
In this context, we associate to $Q$ a graph $G_Q$ with vertex set $[N]$ and connect $i$ and $j$ by an edge, written $i \sim j$, if the conductance $c(ij)$ is nonzero, where the conductance is defined as
\begin{equation}\label{conductance_definition}
c(ij) := \pi(i) Q(i,j) = \pi(j) Q(j,i).
\end{equation}

This is also known as the \emph{ergodic flow} of the edge. Moreover, the interpretation as an electric network lets us define the effective resistance between two vertices $i,j \in [N]$, denoted $\mathcal{R}(i \leftrightarrow j)$, as in \cite[Chapter 9]{wilmer2009markov}.

To state the following proposition, we also define  ${\rm diam}(Q)$ to be the diameter in the graph theoretic 
sense for the graph obtained from $Q$ as above.
This is a standard result, which can be proven using Thomson's principle as in \cite[Proposition 4.4]{fernley2019voter}.

\begin{proposition}\label{prop:max_resistance}
Let  $(X_t)_{t \geq 0}$ be a reversible, irreducible Markov chain on $[N]$ with associated conductances $c$.
Let $P_{i,j}$ be a path from $i$ to $j$ in $G_Q$ and denote by $E(P_{i,j})$ the set of edges in $P_{i,j}$. Then
\[
\mathbb{E}_i \left( T_j \right) +\mathbb{E}_j \left( T_i \right) \leq \sum_{e \in E(P_{i,j})}\frac{1}{c(e)}.
\]

In particular, we have
$ t_{\rm hit} \leq {\rm diam} (Q) \max_{i\sim j \in [N]} \nicefrac{1}{c(ij)}$.
\end{proposition}

There are competing definitions of the distance from stationarity, two of which we need to apply the literature results. 
For a Markov chain on $[N]$ define the usual total variation distance and mixing time at TV threshold $\nicefrac{1}{e}$
\begin{equation}\label{eq_mixing_def}
d_{\rm TV}(t):=\frac{1}{2}\max_{x \in [N]}\| p^{(t)}_{x,\cdot}-\pi(\cdot) \|_1,
\quad
t_{\text{\textnormal{mix}}}:=\min \Big\{ t\geq 0 : d_{\rm TV}(t)\leq \tfrac{1}{e}\Big\}.
\end{equation}

We also quantify mixing by the relaxation time
\[
t_{\text{\textnormal{rel}}}:=\max \left\{ \tfrac{-1}{\lambda} : \lambda \text{ is a nonzero eigenvalue of } Q \right\}. 
\]

If $(Y_t)_{t \geq 0}$ is an independent copy of the chain with arbitrary initial condition, then the random meeting time for the two processes is
\[
\tau_{\rm meet}:=\inf_{t \geq 0} \left\{
t:
X_t=Y_t
\right\}
\]
and the expected meeting time is defined by the worst case initial conditions
\[
t_{\rm meet}:=\max_{x,y \in [N]}
\mathbb{E}\left(
\tau_{\rm meet}
\big|
X_0=x,
Y_0=y
\right).
\]

It will often be easier to work with the (expected) meeting time when both chains are started in 
the invariant measure, i.e.\ we define
\[ t_{\rm meet}^\pi :=  \sum_{i, j \in [N]} \pi(i) \pi(j) \E_{i,j}(\tau_{\rm meet}) .\]

Most importantly, this stationary version has the following useful lower bound.

\begin{proposition}[ {\cite[Remark 3.5]{cox2016convergence}} ]\label{lower_meeting_bound}
\[
t_{\text{\textnormal{meet}}}^{\pi} \geq \frac{(1-\sum_{i \in [N]} \pi(i)^2)^2}{4\sum_{i \in [N]} q(i) \pi(i)^2}, 
\]
where $q(i) = - Q(i,i)$ is the vertex rate of the dual dynamic.
\end{proposition}

We now use the concept of the chain $(X_t)_{t \geq 0}$ \emph{observed on a subset} $V \subset [N]$ described in Section 2.7.1 of \cite{aldous-fill-2014}.

\begin{definition}[Partially observed chain]\label{define_partial_obs}
For a chain $(X_t)_{t \geq 0}$ on $[N]$, take $V \subset [N]$ and define a clock process
\[
U(t):=\int_0^t \mathbbm{1}_{V} \left( X_s \right) {\rm d}s, 
\]
with generalised right-continuous inverse $U^{-1}$. Then the partially observed chain $(P_t)_{t \geq 0}$ is defined for any  $t \geq 0$ via
\[
P_t:=X_{U^{-1}(t)}.
\]

This corresponds to the deletion of states in $V^c$ from the trajectory of $(X_t)_{t\geq 0}$, and thus it can be seen that  $(P_t)_{t \geq 0}$ is Markovian and has the natural stationary distribution
\[
\frac{\pi(\cdot) \mathbbm{1}_{V} (\cdot)}{\pi(V)}.
\]
\end{definition}

We can define the random \emph{subset meeting time} $\tau_{\text{\textnormal{meet}}}(A)$ analogously to $\tau_{\text{\textnormal{meet}}}$ except for the partially observed product chain on $A \times A$ rather than the full chain. Similarly,  $t^\pi_{\text{\textnormal{meet}}}(A)=\mathbb{E}_{\pi \otimes \pi}(\tau_{\text{\textnormal{meet}}}(A))$.

\subsection{Mixing Time for the Variable Speed Random Walk}\label{section_erw_mixing}

Our approach is inspired by the structural theorems in \cite{fountoulakis2008evolution} and \cite{benjamini2014mixing} which led to mixing time bounds for the Erd\H{o}s-R\'enyi giant component, and the ``decorated expander'' of \cite{ding2014anatomy} who work in discrete time.

The mixing time is bounded by finding an Erd\H{o}s-R\'enyi giant component as a subgraph and growing it to a subgraph spanning the giant component of $G_N$. To find this subgraph, then, we require the artificial condition $\beta>1$.
Transitions in the edge density $\beta$, excluding those caused by the appearance of the giant component, are unusual for interacting particle systems on networks and so one might argue that the large $\beta$ condition is not an important omission. Given that our condition is numerically seen to be satisfied for any $\gamma \in (0,1)$ when $\beta>\nicefrac{5}{4}$, it also doesn't omit very much.

The VSRW of Definition \ref{def_vsrw} is the dual Markov chain for the discursive voter models when $\theta=1$.
A monotonicity in $\theta$ of these models will allow all $\theta \geq 1$ mixing results to follow from a bound on the mixing time of the variable speed random walk on the SNR graph, and hence most of the chapter will be dedicated to proving that bound in the VSRW case. Of course the variable speed random walk is a natural model and so this bound has additional independent interest.

\begin{definition}[Cheeger constant]
For any connected graph $G$ on $[N]$ we define the (VSRW) Cheeger constant
\[
\Phi(G) := \inf_{\stackrel{S \subset [N]}{|S|\leq \nicefrac{N}{2}}}\frac{e \left( S: S^C \right)}{\left| S \right| }
\]
where $e \left( A: B \right)$ denotes the number of edges incident to a vertex in both $A$ and $B$, i.e. the number of edges between sets $A$ and $B$.
\end{definition}

There are many parametrisations leading to slightly different versions of the below result; here is the version for the choice above applied to VSRW.

\begin{proposition}\label{theorem_cond_and_relax}
The VSRW relaxation time for a connected graph $G$ on $[N]$ has
\[
\frac{1}{2 \Phi(G)}
\leq
t_{\rm rel}
\leq
\frac{8}{\Phi(G)^2}
\max_{v \in [N]} \de(v)
\]
\end{proposition}

\begin{proof}
Cheeger's inequality \cite[Theorem 4.40]{aldous-fill-2014} gives
\[
t_{\rm rel}
\leq
8 \tau_c^2 \max_{v \in [N]} \de(v)
\]
for the different parameter $\tau_c$ which has
\[
\tau_c:= \sup_{A \subset [N]}
\frac{|A|}{e(A:A^c)}
\cdot
\frac{N-|A|}{N}
\leq
\sup_{\stackrel{A \subset [N]}{
|A|\leq \frac{1}{2}}}
\frac{|A|}{e(A:A^c)}
=\frac{1}{\Phi(G)}.
\]

For the other direction, this $\tau_c$ definition is symmetric so we can assume w.l.o.g. that $|A|\leq N-|A|$. Then \cite[Corollary 4.37]{aldous-fill-2014} completes the proof:
\[
t_{\rm rel}
\geq
\tau_c
\geq
\frac{1}{2}
\sup_{A \subset [N]}
\frac{|A|}{e(A:A^c)}
=\frac{1}{2\Phi(G)}.
\]
\end{proof}

\begin{definition}
The Erd\H{o}s-R\'enyi graph on $[N]$ with parameter $\beta$ has independent edges with homogeneous probabilities $p_{ij}=\nicefrac{\beta}{N}$ for every unordered pair $i \neq j$.
\end{definition}

\begin{lemma}\label{lemma_er_conductance}
If $G$ is the largest component of an Erd\H{o}s-R\'enyi graph on $[N]$ with parameter $\beta>1$ then
\[
\Phi(G) =
\Omega_{\mathbb{P}}
\left(
\frac{1}{\log^2 N}
\right)
\]
\end{lemma}

\begin{proof}
From \cite[Theorem 1.2]{fountoulakis2008evolution} we have \emph{for the constant speed random walk (CSRW)}, the simple random walk with steps at constant Poisson rate $1$,
\[
t^{\rm C}_{\rm mix} \left( G \right)
=O_{\mathbb{P}}\left( \log^2 N \right)
\]
and \cite[Lemma 4.23]{aldous-fill-2014} translates this to the same bound on $t_{\rm rel} \left( G \right)$. Then \cite[Corollary 4.37]{aldous-fill-2014} applies to $\tau_c$ for the CSRW which has form
\[
\tau^{\rm C}_c
:=
\sup_{A \subset [N]}
\frac{\pi(A)\pi(A^c)}{\sum_{x \in A, y \notin A}\pi(x)Q(x,y)}
=
\frac{1}{\de([N])}
\sup_{A \subset [N]}
\frac{\de(A)\de(A^c)}{e(A:A^c)}
\]
and so we can deduce a bound on VSRW conductance from CSRW mixing
\[
t^{\rm C}_{\rm mix}
\geq
\tau^{\rm C}_c
\geq
\frac{1}{2}
\sup_{A \subset [N]}
\frac{\de(A)}{e(A:A^c)}
\geq
\frac{1}{2}
\sup_{A \subset [N]}
\frac{|A|}{e(A:A^c)}
=\frac{1}{2\Phi(G)}
\]
which after inversion is the desired bound.
\end{proof}

We need to explore the graph $G_N$ through its local tree approximation, and the following useful algorithm comes from \cite[Proposition 3.1]{norros2006conditionally} applied to our graph parameters.

\begin{lemma}\label{lemma_tree_exploration}
From any subset $V \subset [N]$ with induced subgraph $G_V$, we can grow the subgraph to $G_N$ in the following way:
\begin{itemize}
\item Fix an arbitrary ordering of $V$ and root $\circ \in V$.
\item In first this order and then in breadth-first order from $\circ$, select some  vertex $i\in[N]$:
\begin{itemize}
\item \emph{Explore} $i$ by giving it putative offspring drawn independently from $\operatorname{Pois}(w(i))$, with the weight
\[
w(i):=\beta N^{2\gamma-1}i^{-\gamma} \sum_{j=1}^N j^{-\gamma}.
\]
\item \emph{Label} each of these putative offspring with an independent label drawn from the mark distribution $M$ with
\begin{equation}\label{eq_mark_dist}
\mathbb{P}(M=k) \propto k^{-\gamma}\mathbbm{1}_{k \in [N]}.
\end{equation}
\item \emph{Thin} by deleting any of these offspring which shares a label with an explored vertex.
\item Any offspring with labels shared by an unexplored vertex are \emph{cycle edges} and the offspring should be identified with the other vertex sharing its label.
\end{itemize}
\end{itemize}
\end{lemma}
The edges which lead to offspring not deleted are exactly distributed as the edges of the SNR graph $G_N\setminus G_V$.

We can now prove the main result of this section bounding the VSRW mixing time on an SNR network. 
This proof has the following essential structure:
\begin{itemize}
\item Use that $\beta>1$ to find an Erd\H{o}s-R\'enyi giant component as a subgraph of the giant component of $G_N$, which is known to be fast mixing;
\item Grow it to a spanning subgraph of the giant of $G_N$ by attaching trees of size $O^{\log N}_{\mathbb{P}}\left( 1 \right)$, argue that the resultant structure also has fast mixing;
\item Because the VSRW has a symmetric generator, completing the internal edges can only accelerate relaxation time and so we have controlled the mixing on the component of interest.
\end{itemize}

\erwmixing*

\begin{proof}
We make here Assumption \ref{assumption_large_beta}, and show later in Proposition \ref{prop_beta_3} that this assumption is given by $\beta \geq 3$ and in Proposition \ref{prop_beta_2log} by $\beta > 2\log 2$ if also $\gamma \geq \nicefrac{1}{2}$ (note that in either case we have $\rho>\nicefrac{1}{2}$ whenever $\beta>2 \log 2$).

Recall that the MNR network is constructed by the weight profile 
$
\operatorname{f} : x \mapsto \sqrt{\beta} x^{-\gamma} 
$ 
 on $(0,1]^2$ in that there are $\operatorname{Pois}\left(\frac{1}{N} \operatorname{f}\left(\frac{v}{N}\right)\operatorname{f}\left(\frac{w}{N}\right)\right)$ edges between each pair of distinct vertices $v,w \in [N]$. Thinking of the graph then as a Poisson point process on $(0,1]^2$, we can decompose its mass into a sum of constituent parts which we will realise independently to construct a subgraph of the MNR network.

For some small $\epsilon>0$, we distinguish the vertex set of high weight vertices
\[
H_\epsilon=\left[ \left\lfloor \epsilon N \log^{-\frac{1}{\gamma}} N \right\rfloor \right]
\]
and assume that $\beta>1$ so that we can use the constant part $(x,y) \mapsto \beta$ to construct an Erd\H{o}s-R\'enyi graph $\sN_{\rm ER}$ with edge probabilities $\nicefrac{\beta}{N}$ on the set $[N]\setminus H_\epsilon$. In fact, this part generates the Poissonian Erd\H{o}s-R\'enyi  graph $\sN_{\rm PER}$ with edge probabilities 
$
1-e^{\nicefrac{-\beta}{N}} 
$ 
but by \cite[Theorem 6.18]{van2016random} we have asymptotic equivalence 
\[
\mathbb{P}\left(
\sN_{\rm ER} \neq \sN_{\rm PER}
\right)
\rightarrow 0.
\]

In this Erd\H{o}s-R\'enyi graph, realise the largest component $\sG_{\rm ER}$ on a vertex set denoted $\sC_{\rm ER}$. By constructing the largest component, we have conditioned that the other components are smaller. Remove the $o(N)$ of these vertices in $H_\epsilon$ to leave
\[
\sC_1=\sC_{\rm ER} \setminus H_\epsilon.
\]

We note further, as $|H_\epsilon|=o(N)$, that
\[
\frac{\left| \sC_1 \right|}{N} \stackrel{\mathbb{P}}{\rightarrow} \rho
\]
where $\rho \in (0,1)$ is the unique solution to $\rho +e^{-\beta \rho}=1$ (see e.g. \cite[Theorem 3.1]{bollobas2007phase}). Hence this is the giant component, and we find for any small $\delta\in \left(0,\nicefrac{1}{2}\right)$ that $|\sC_1|\geq (1-\delta)\rho N$ with high probability. 
By using that
\begin{equation}\label{eq_betarho}
1-\rho=e^{-\beta \rho}<\frac{1}{1+\beta\rho} \implies \beta(1-\rho)<1,
\end{equation}
we can take $\delta$ sufficiently small such that $\beta(1-(1-\delta)\rho)<1$ which leads to, on the event $\{|\sC_1|\geq (1-\delta)\rho N\}$, exploration of the other components in $[N]\setminus H_\epsilon$ being dominated by subcritical Poisson-Galton-Watson trees. These subcritical trees have maximal size $O_{\mathbb{P}}(\log N)$. Therefore, on this first conditioning $|\sC_1|\geq (1-\delta)\rho N$ which is a high probability event, the further event that they do not form a larger component than $|\sC_1|$ occurs with high probability. Thus conditioned and unconditioned models are asymptotically equivalent and so we can work with independent edges outside $\sC_1$ in the remainder of the proof (and more, by another appeal to asymptotic equivalence, return to generating these edges with the Poisson probabilities $1-\exp(\nicefrac{-\beta}{N})$).

The next extension to the subgraph is to attach the set $H_\epsilon$.
Note between any $v \in \sC_1$ and $h \in H_\epsilon$ we have a Poisson number of edges with parameter at least $\beta \left( \epsilon^{-\gamma} \log N\right)$.

$H_\epsilon$ is attached to $\sC_1$ vertex by vertex: construct a function $n : H_\epsilon \rightarrow \sC_1$ iteratively by giving the lowest unpaired index $h \in H_\epsilon$ its lowest neighbour in $\sC_1 \setminus n \left( [h-1] \cap H_\epsilon\right)$ which we set as $n(h)$, and connect $h$ to $\sC_1$ by this single edge. Because $|H_\epsilon|<\delta \rho N$ deterministically and $|\sC_1|>(1-\delta) \rho N$ with high probability, we have (whp) at least $(1-2\delta) \rho N$ available vertices in $\sC_1$ to pair to at any stage of this iteration. The mean number of edges to available vertices that each $h \in H$ will see is thus at least
\[
(1-2\delta) \rho N \cdot \frac{1}{N}\beta \epsilon^{-\gamma} \log N
> \beta \log N,
\]
by taking $\epsilon$ small enough such that $(1-2\delta) \rho \epsilon^{-\gamma}>1$.

We observe $\mathbb{P}\left(\operatorname{Pois}(\beta \log N)=0 \right)=N^{-\beta}=o\left(\nicefrac{1}{N}\right)$ and conclude by the union bound that with high probability we are successful in constructing this injective function $n : H \rightarrow \sC_1$. The graph $\sG_2$ is $\sG_1$ with every vertex in $H_\epsilon$ attached as a leaf in this way, on the vertex set $\sC_2=\sC_1 \cup H_\epsilon$.

Then on this same vertex set $\sC_2$ we realise the remaining Poisson kernel to make this the \emph{induced} subgraph on $\sC_2=:\sC_3$. This means that between any pairs $\{v,w\}$ of either the form $v,w \in \sC_1$ or the form
\[v \in H_\epsilon, 
w \in n \left( [v-1] \cap H_\epsilon \right)
\cup
\left(
\sC_1 \setminus [n(v)]
\right)
\cup H_\epsilon,
\]
we have an independent Poisson number of edges with  mean parameter 
$
\frac{\beta}{N}\left(N^{2\gamma} v^{-\gamma} w^{-\gamma}-1 \right).
$ 
Flatten multiple edges and call the resultant simple graph $\sG_3$.

At this point we should begin to discuss the mixing times. We start with Lemma \ref{lemma_er_conductance} which lower bounds the Cheeger constant of $\sG_1$
\[
\Phi(\sG_1) =
\Omega_{\mathbb{P}}
\left(
\frac{1}{\log^2 N}
\right).
\]

Note that a set of minimal Cheeger constant for $\sG_2$ is simply a connected subset of $\sC_1$ with its pendant leaves included. 
Because each vertex in $H$ is attached to a distinct vertex in $\sC_1$, the worst case is that every one gets a pendant edge and hence $\Phi(\sG_2)\geq \Phi(\sG_1)/2$. 
By Proposition \ref{theorem_cond_and_relax} we deduce
\[
t_{\rm rel} \left( \sG_2 \right)
\leq
\frac{8 }{\Phi^2\left( \sG_2 \right)}
\max_{v \in \sC_2} \de_{\sG_2}(v)
\leq
\frac{32 }{\Phi^2\left( \sG_1 \right)}
\max_{v \in \sC_1}(1+ \de_{\sG_1}(v))
= O\pl\left( \log^{5} N \right).
\]

For the third graph $\sG_3$ we claim by \cite[Corollary 3.28]{aldous-fill-2014} that adding the internal edges did not increase the relaxation time, and so we have the same bound $t_{\rm rel}(\sG_3) \leq t_{\rm rel}(\sG_2)$. By applying \cite[Lemma 4.23]{aldous-fill-2014} we can turn this into a bound on the mixing time
\begin{equation}\label{eq_g3_mixing}
t_{\rm mix}(\sG_3)
\leq
t_{\rm rel}(\sG_3)\left( 1+ \frac{1}{2} \log \left| \sC_3 \right| \right)
=O\pl\left( \log^{6} N \right).
\end{equation}

$\sG_3$ is an \emph{induced} subgraph which will form the fast-mixing core. 
The rest of the construction is by adding pendant trees to create a spanning subgraph which contains this fast-mixing core. 

To this end, in $\sG_5$, we want to not create any additional cycles -- we want to span the rest of the giant component only by growing trees. Therefore we will explore the neighbourhood of every vertex in $\sC_4 \setminus \sC_3$ while thinning any label in $\sC_4$, and any labels previously seen in this $\sG_5$ construction. 
This exploration can be done with the usual thinned Galton-Watson exploration of Lemma \ref{lemma_tree_exploration} using an independently drawn label from the distribution \eqref{eq_mark_dist} for each vertex, and just skipping the step concerning cycle edges which are instead deleted.   
In Lemma \ref{g4_subcrit}, later, we will argue that these explorations are subcritical and hence maximally $O_{\p}(\log^3 N)$.

From this graph $\sG_5$ of a core $\sG_3$ with pendant trees, let $\sG_4$ be the graph $\sG_3 \leq \sG_4 \leq \sG_5$ containing the ball of radius $1$ around $\sG_3$ -- i.e. only the first vertex of each pendant tree.

Recall now Definition \ref{define_partial_obs} of the partially observed chain.
Because $\sG_5 \setminus \sG_3$ is composed of pendant subtrees which attach to $\sC_3$ at a single vertex, the walker leaves from and returns to $\sC_3$ at that same vertex. Hence, 
the VSRW on $\sG_5$ partially observed on $\sC_3$ has the same dynamic as the VSRW on $\sG_3$ which we recall from \eqref{eq_g3_mixing} had mixing time $O_{\mathbb{P}}\left( \log^{6} N \right)$.

Define the $\sC_3$ occupancy clock of a walker $(W_t)_t$
\[
\sigma_{\sC_3}(t):=\int_0^t \mathbbm{1}_{W_s \in \sC_3} {\rm d}s
\]
so that what we meant above precisely is that, by \cite[Theorem 1.1(b)]{fill1991time}, we can construct a strong stationary time $T$ with 
\[
\p\left( W_{\sigma_{\sC_3}(T)}=v\right)= \frac{\mathbbm{1}_{v \in\sC_3}}{|\sC_3|}
\]
and then by \cite[Lemma 6.17]{wilmer2009markov} and \cite[Lemma 4.5]{aldous-fill-2014}
\begin{equation}\label{eq_T_tail}
\mathbb{P}(\sigma_{\sC_3}(T)>t)\leq 4 d^{(\sG_3)}_{\rm TV}\left(\tfrac{t}{2}\right)\leq
4e^{- \left\lfloor\tfrac{t}{2 \, t_{\rm mix}(\sG_3)}\right\rfloor}
\end{equation}
where $d^{(\sG_3)}_{\rm TV}$ is our notation for the usual worst-case total variation distance but for the VSRW on $\sG_3$. 
By Lemma \ref{lemma_log_degrees} for the small weights, and from an easy argument with Poisson large deviations for the others,
\begin{equation}\label{eq_degree_handwaving}
\max_{v \in [N]}
\frac{\de(v)}{\left( \frac{N}{v} \right)^\gamma}
=
O_{\mathbb{P}}(\log N)
\end{equation}
which we combine with the observation $\de_{\sG_5}(v)\leq \de(v)$ and Lemma \ref{lemma_high_internal_edges} after this proof (taking $\epsilon$ small enough to give $\beta \rho \epsilon^{-\gamma}>32$) to deduce that vertices $v \in \sC_3$ have neighbourhoods $\Gamma(v)$ with
\[
\min_{v \in \sC_3} \frac{\left|\Gamma(v)\cap \sC_1\right|}{\de_{\sG_5}(v)}
=
\Omega_{\mathbb{P}} \left(
\frac{1}{\log^2 N}
\right).
\]

For the VSRW on this graph $\sG_5$, the maximal expected time to escape a pendant tree (which we see in the proof of Lemma \ref{g4_subcrit} is maximally of size $O_{\mathbb{P}}\left( \log^{3} N \right)$) is $O_{\mathbb{P}}\left( \log^{6} N \right)$ by Proposition \ref{prop:max_resistance}.

Therefore, by taking geometrically many attempts to escape vertices in $H_\epsilon$, we conclude the maximal expected time to hit the set $\sC_1$ (from any initial vertex in  $\sC_{\rm max}$ which is the vertex set of $\sG_5$) is $O_{\mathbb{P}}\left( \log^{8} N \right)$.

We expect to hit $\sC_1$ in maximal time $t_{\rm hit}(\sC_1)=O_{\mathbb{P}}\left( \log^{8} N \right)$, and by Lemma \ref{lemma_log_degrees} we expect to stay in that vertex for an exponential waiting period of mean at least 
\[
a=
\min_{v \in \sC_1}
\frac{1}{\de_{\sG_5}(v)}=
\Omega_{\mathbb{P}}\left( \frac{1}{\log N} \right).
\] 

By Markov's inequality, we have at least probability $\nicefrac{1}{2}$ to hit before time $2t_{\rm hit}(\sC_1)$: if so we wait at least time $a$ with probability $\nicefrac{1}{e}$ and then in either case we insert another hitting time to bring the walker back to $\sC_1$. 
Therefore we have a Chernoff bound
\[
\p\left(
\sigma_{\sC_3}(2t_{\rm hit}(\sC_1)\left\lceil 3 e k \right \rceil +ak)\leq ak
\right)
\leq\p\left(
\operatorname{Bin}\left(\left\lceil 3 e k \right \rceil, \frac{1}{2e} \right) \leq k
\right)\leq e^{\nicefrac{-k}{12}}
\]
or, more loosely,
\[\p\left(
\sigma_{\sC_3}(t)\leq a\left\lfloor \frac{t}{20 \,  t_{\rm hit}(\sC_1)}\right\rfloor
\right)
\leq \exp\left(-\frac{t}{240 \, t_{\rm hit}(\sC_1)}\right)\]
(recall also $\nicefrac{a}{t_{\rm hit}(\sC_1)}=\Omega_{\mathbb{P}}\left( \nicefrac{1}{\log^9 N} \right)$).
Together with Equation \eqref{eq_T_tail}, this provides for large constants $C,C'>0$ and small $\delta>0$
\[
\p \left( T>t  \right)
\leq
\p \left( \sigma_{\sC_3}(T)>C \log^6 N  \right)
+
\p \left( \sigma_{\sC_3}(t)\leq C \log^6 N  \right)\leq \frac{\delta}{2}
\]
at $t=C'\log^{15}N$. 
We then argue that as $d^{(\sG_5)}_{\rm TV}(t)$ is the least failure probability of a coupling to $\pi_{\sC_5}$ at time $t$, 
\[
\begin{split}
d^{(\sG_5)}_{\rm TV}(t) \leq \p \left( T\leq t  \right) \frac{1}{2}\left\| \frac{\mathbbm{1}_{\sC_3}}{|\sC_3|}-\frac{\mathbbm{1}_{\sC_5}}{|\sC_5|} \right\|_1 + \p \left( T>t  \right)
&\leq
1-\frac{|\sC_3|}{|\sC_5|}+ \p \left( T>t  \right)\\
&\leq
1-\left(1-\frac{\delta}{2}\right)\rho
\end{split}
\]
on the high probability event $\{|\sC_1|\geq (1-\delta)\rho N\}$. By Assumption \ref{assumption_large_beta} we have $\rho>\nicefrac{1}{2}$, and so we have a mixing time on the same order as $t$ using \cite[Lemma 2.20]{aldous-fill-2014} to say
$
\de(s)\leq
\left(2\de(t)\right)^{\left\lfloor\nicefrac{s}{t}\right\rfloor}
$.

Note 
$t_{\rm rel}(\sG_5)
\leq
t_{\rm mix}(\sG_5)$
and complete the remaining edges internal to the vertex set $\sC_{\rm max}$ of this graph. Thus the resultant graph $\sG_{\rm max}$ is the SNR giant and by another application of \cite[Corollary 3.28]{aldous-fill-2014} these extra edges cannot slow relaxation time, and finally by \cite[Lemma 4.23]{aldous-fill-2014} we bound the mixing time on the order $O_{\mathbb{P}}\left( \log^{16} N \right)$. 
\end{proof}

\begin{figure}[hp]
\thisfloatpagestyle{empty}
\begin{center}
\hspace*{\fill}
\subfloat[$\sG_1$ is the Erd\H{o}s-R\'enyi core with some $H_\epsilon$ vertices removed]{\includegraphics[scale = .25]{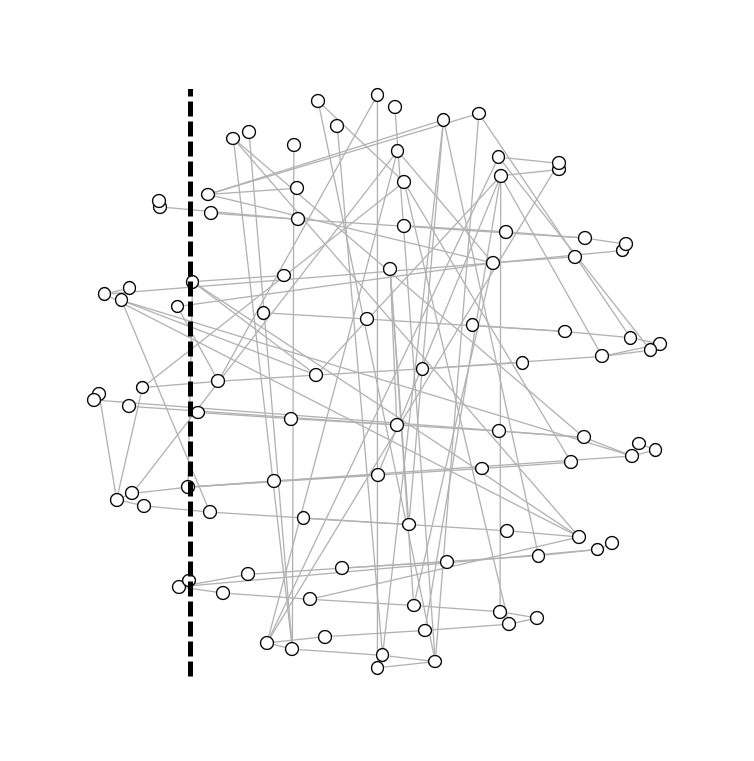}}
\hspace*{\fill} \\
\hspace*{\fill}
\subfloat[$\sG_2$ connects all of $H_\epsilon$]{\includegraphics[scale = .25]{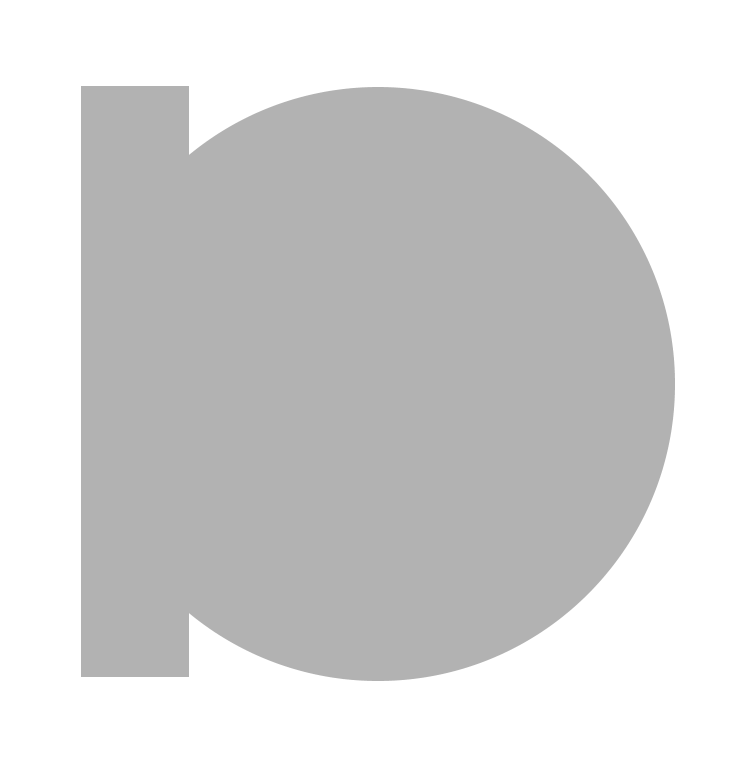}}
\hfill
\subfloat[$\sG_3$ completes internal edges]{\includegraphics[scale = .25]{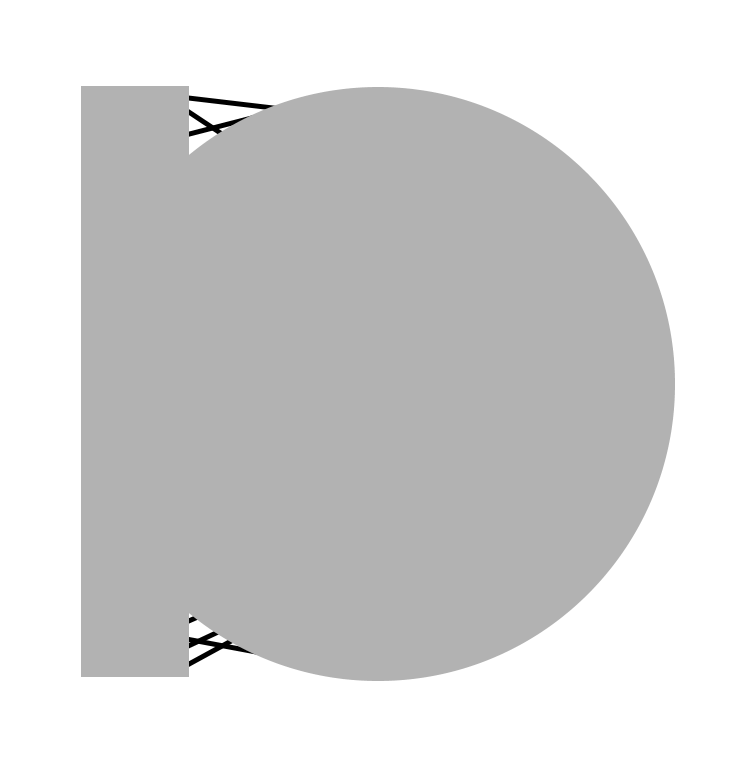}}
\hspace*{\fill} \\
\hspace*{\fill}
\subfloat[$\sG_4$ explores a ball]{\includegraphics[scale = .25]{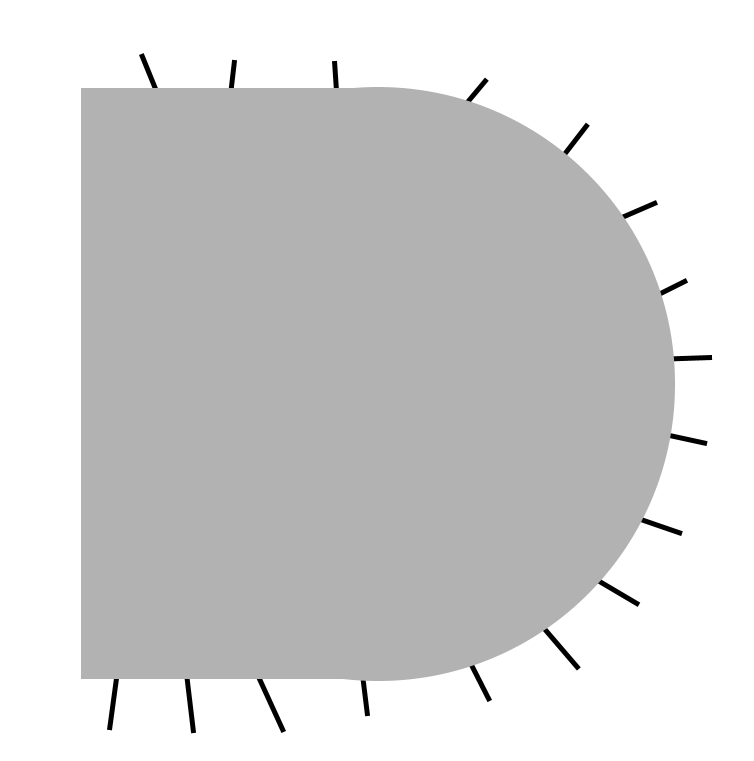}}
\hfill
\subfloat[$\sG_5$ completes the subcritical trees to span the full giant]{\includegraphics[scale = .25]{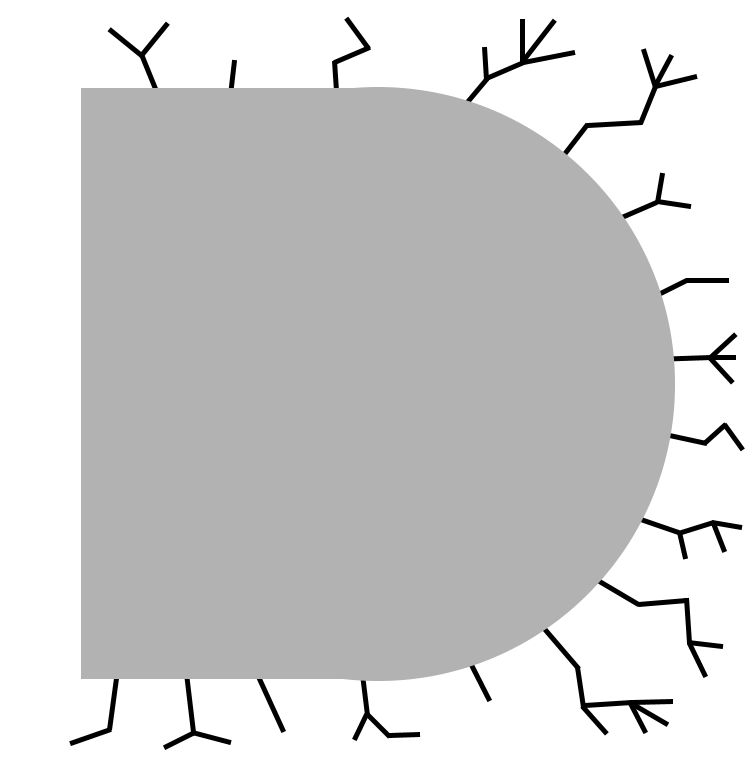}}  
\hspace*{\fill} \\
\caption{We sketch the layers of the spanning subgraph that we construct in this proof. The first layer, $\sG_1$, was necessary to have a fast-mixing core to build from; we aim to reach the point where the unexplored mass outside of the subgraph $\sG_4$ has $\sum_{v \notin \sG_4}\frac{\beta}{N}\left( \frac{N}{v} \right)^{2\gamma}<1$ so that the final growths will be subcritical. Further to keep these trees small, we must have bounded weights (in the sense of a $O^{\log N}(1)$ bound) which we achieve outside $\sG_2$. $\sG_3$ is a minor step to keep track of mixing, and then in $\sG_4$ we use the large $\beta$ assumption to check that we have definitely passed the point of subcritical explorations. The graph $\sG_5$ adds the promised subcritical trees which then spans the full giant $\sC_{\rm max}$.}\label{fig:subgraph_construction}
\end{center}
\end{figure}

We omitted two claims in the previous proof -- in the Appendix we check that $H_\epsilon$ was well connected to the Erd\H{o}s-R\'enyi giant, and here we show that with high probability every pendant tree was uniformly bounded by a polylogarithm.

\begin{lemma}\label{g4_subcrit}
In the above construction, the largest component of $\sG_5 \setminus \sG_3$ has size $O_{\mathbb{P}}\left( \log^3 N \right)$.
\end{lemma}

\begin{proof}
Each of these components is a tree by construction, and each gains exactly one edge from the construction of $\sG_4$.

Afterwards, we complete the exploration by the thinned Galton-Watson exploration of Lemma \ref{lemma_tree_exploration} where as we explore new vertices we only have more labels to thin from the future exploration. Hence, we can simplify the exploration by stochastically containing every pendant tree in a tree with just labels in $\sC_4$ thinned (these trees are then i.i.d.).

To be clear, we explore such a tree in the following way.
Generate a random label $M \in [N]$ with $\mathbb{P}(M=k) \propto k^{-\gamma}$, the effective weight is then only positive if $M \notin \sC_4$ and so is the following function of $M$:
\[
W (M) =
\begin{cases}
\beta N^{2\gamma-1} M^{-\gamma} \sum_{k=1}^N k^{-\gamma}
& M \notin \sC_4 \\
0 & M \in \sC_4. \\
\end{cases}
\]

Offspring have the mixed Poisson distribution $D \sim \operatorname{Pois} \left( W \right)$ and thus we have a Galton-Watson tree. Note that rather than removing thinned vertices we have given them zero weight and hence no children -- for the purposes of containing the tree this is sufficient.

Think of this network exploration in continuous time such that each vertex $i \in [N]$ is revealed to be the next unrevealed vertex (say, in the breadth-first order) as a Poisson process of rate $\left(\nicefrac{N}{i}\right)^{\gamma}$. Thus all times after the first for each Poisson process will represent thinned vertices in the tree construction. Note also that the total exploration rate is
\begin{equation}\label{eq_explore_rate}
\sum_{i=1}^N 
\left(\frac{N}{i}\right)^{\gamma}
\sim
\frac{N}{1-\gamma}.
\end{equation}

$\sG_4$ was constructed as the ball around $\sC_3=\sC_1\cup H_\epsilon$, where around $\sC_1$ this is with the remaining kernel $\beta (x^{-\gamma}y^{-\gamma}-1)$ and around $H_\epsilon \setminus V(\sG_1)$ we have the full kernel $\beta x^{-\gamma}y^{-\gamma}$. In the MNR model the number of edges in this ball is a single concentrated Poisson variable and so we build this ball by exploring edges numbering at least
\[
\begin{split}
&(N+o_{\mathbb{P}}(N))
\int_0^1 \int_0^1 
\rho \beta (x^{-\gamma}y^{-\gamma}-1)
{\rm d}x {\rm d}y\\
=&(N+o_{\mathbb{P}}(N))
\rho \beta \left( \frac{1}{(1-\gamma)^2}-1 \right),
\end{split}
\]
note that still each of these edges is thinned if it finds a repeated label. Attaching these labels in continuous time at the rate \eqref{eq_explore_rate} then takes time $T+o_{\mathbb{P}}(1)$ for
\[
T:=
\rho \beta \left( \frac{1}{1-\gamma}-1+\gamma \right).
\]

We can therefore take some small $\delta>0$ and thin from $[N]\setminus \sC_1$ in continuous time for time $(1-\delta)T$, and on the high probability event that the continuous time exploration takes more than time $(1-\delta)T$ this is an upper bound on the remaining mass by stochastic domination.

Each unthinned vertex has a contribution to the offspring mean of the ongoing Galton-Watson exploration given by
\[
\mathbb{P}(M=v)w(v)=
\frac{v^{-\gamma}}{\sum_{k=1}^N k^{-\gamma}}
\sum_{k=1}^N \beta N^{2\gamma-1} v^{-\gamma} k^{-\gamma}
= \frac{\beta}{N}
\left( \frac{N}{v} \right)^{2\gamma}.
\]

Define also a high probability lower bound on the set $\sC_1$, which is just a set $S$ including each vertex from $[N]$ independently with probability $(1-\delta)\rho$. Then on both high probability assumptions we have mass at time $(1-\delta)T$ stochastically dominated by
\[
\mathfrak{M}
=
\sum_{v=1}^N
\frac{\beta}{N}
\left( \frac{N}{v} \right)^{2\gamma}
\mathbbm{1}_{E_v>(1-\delta)T}
\mathbbm{1}_{v \notin S}
\]
where each $E_v$ is an independent exponential of rate $\left(\nicefrac{N}{v}\right)^{\gamma}$. We can bound this variable with the second moment method, and so we calculate
\[
\begin{split}
\operatorname{Var}\left(
\mathfrak{M}
\right)
&=\sum_{v=1}^N
\frac{\beta^2}{N^2}
\left( \frac{N}{v} \right)^{4 \gamma}
\left(1-e^{- (1-\delta)T\left( \frac{N}{v} \right)^{\gamma}}\right)e^{- (1-\delta)T\left( \frac{N}{v} \right)^{\gamma}}(1-\delta)\rho(1-(1-\delta)\rho)\\
&\leq \sum_{v=1}^N
\frac{\beta^2}{N^2}
\left( \frac{N}{v} \right)^{4 \gamma}
e^{- (1-\delta)T\left( \frac{N}{v} \right)^{\gamma}}
\sim \frac{1}{N} \int_0^1 \beta^2 x^{-4\gamma} e^{- (1-\delta)T x^{-\gamma}} {\rm d} x
\rightarrow 0 \text{ as } N \rightarrow \infty
\end{split}
\]
because the integrand is bounded. Hence $\mathfrak{M}$ converges in probability to its mean:
\[
\begin{split}
\mathbb{E}(\mathfrak{M})
&=
\sum_{v=1}^N
\frac{\beta}{N}
\left( \frac{N}{v} \right)^{2\gamma}
e^{- (1-\delta)T \left( \frac{N}{v} \right)^{\gamma}}
(1-(1-\delta)\rho)\\
&\rightarrow
\beta(1-(1-\delta)\rho)
\int_0^1
x^{-2\gamma}
e^{-(1-\delta)T x^{-\gamma}}
{\rm d} x\\
&<
\int_0^1
x^{-2\gamma}
e^{-(1-\delta)T x^{-\gamma}}
{\rm d} x \text{ for small $\delta$ by \eqref{eq_betarho}. }
\end{split}
\]

This integral is less that $1$, recalling that $1-\rho=e^{-\beta \rho}$ and taking $\delta$ small enough, by Assumption \ref{assumption_large_beta}.

So now that we have found some $m<1$ such that with high probability $\mathfrak{M}<m$, we can analyse the trees defined by $\mathfrak{M}$. To each pendant tree $T_i$ we associate an exploration walk $X_i : \mathbb{N}^+ \rightarrow \mathbb{N}$ as in \cite{alon2004probabilistic} which represents the number of unexplored half-edges in the tree if we explore in the breadth-first ordering. Thus $X_1=1$, increments have the i.i.d. distribution
\[
\forall t \geq 1, \quad X_i(t+1)-X_i(t) \stackrel{({\rm d})}{=} D-1
\]
and the first hitting time of $0$ is the size of the tree, $X_i(|T_i|)=0$. Therefore if we fix some large constant $C>0$ and $L=\left\lceil C \log^3 N \right\rceil$, and write $W_i \stackrel{\rm i.i.d.}{\sim}W$
\[
\begin{split}
\mathbb{P}\left( |T_i| \geq L \right)
&\leq
\mathbb{P}\left( X_i \left( L  \right) \geq 0 \right)\\
&=\mathbb{P}\left( \sum_{i=1}^L \operatorname{Pois} \left( W_i \right) \geq L \right)
=\mathbb{P}\left(  \operatorname{Pois} \left( \sum_{i=1}^L W_i \right) \geq L \right)\\
&\leq \mathbb{P}\left(\frac{1}{L} \sum_{i=1}^L W_i \geq \frac{1+m}{2} \right)
+\mathbb{P}\left( \operatorname{Pois} \left(  \left( \frac{1+m}{2} \right) L \right) \geq L \right).\\
\end{split}
\]

Note that the second term is a Poisson large deviation and so has exponential decay in $L$. For the first, we observe that $W$ is a bounded random variable
\[
W \in \left[
0,
\frac{\beta}{1-\gamma} \log N
\right] \quad \text{a.s.}
\]
which will allow us to use Hoeffding's inequality \cite{hoeffding1963} to find
\[
\mathbb{P}\left(\frac{1}{L} \sum_{i=1}^L W_i \geq \frac{1+m}{2} \right)
=
\mathbb{P}\left(\frac{1}{L} \sum_{i=1}^L W_i - m \geq \frac{1-m}{2} \right)
\]
\[
\leq
\exp \left( -
\frac{2 L^2 \left( \frac{1-m}{2} \right)^2}{L \left( \frac{\beta}{1-\gamma} \log N \right)^2}
\right)
\leq
\exp \left( -
\frac{2 C \left( \frac{1-m}{2} \right)^2}{ \left( \frac{\beta}{1-\gamma}  \right)^2} \log N
\right).
\]

So if we set $C$ large enough that
\[
\frac{2 C \left( \frac{1-m}{2} \right)^2}{ \left( \frac{\beta}{1-\gamma}  \right)^2}>1
\]
then no exploration will see $L$ vertices with high probability, by the union bound.
\end{proof}

\section{Discursive Voter Models}

In this final section we put together the previous results to prove Theorems \ref{theorem_disc_cons_small} and \ref{theorem_disc_cons_ultrasmall}.

The VSRW Markov chain is the dual of the discursive voter model with $\theta=1$. Fortunately, controlling the mixing of this dynamic induces a bound over every $\theta \geq 1$ version.

\begin{proposition}[{\cite[Corollary 3.28]{aldous-fill-2014}} ]\label{prop_mix_monotone}
On any connected graph the discursive dual relaxation time $t_{\rm rel}^{(\theta)}$ is non-increasing in $\theta$.
\end{proposition}

\begin{proof}[Proof of Theorems \ref{theorem_disc_cons_small} and \ref{theorem_disc_cons_ultrasmall}]\,


{\bf Case 1:} We first consider $\theta  \leq \frac{1}{\gamma}$.

In either theorem's conditions we have at least $\beta>(1-2\gamma)\vee 0$. With high probability, Lemma \ref{mypaths} gives a set of paths connecting every pair in
\begin{equation*}
S:=
\left\lbrack \left\lfloor
N
\log^{\nicefrac{-2}{\gamma}} N
\right\rfloor \right\rbrack
\end{equation*}
by paths in $[N]\setminus S$, of maximal length $O_{\mathbb{P}}(\log N)$.

The diameter result in Theorem \ref{my_diameter_bound} gives a set of paths between every pair of vertices in $\sC_{\rm max}$, of maximal length $O_{\mathbb{P}}(\log N)$.

Therefore we can alter paths in the second set which have more than $2$ vertices in $S$, by taking the first and last vertex in $S$, deleting the path between them and replacing it by their low degree path from the first set.
From Lemma \ref{lemma_log_degrees} we have 
\[
\max_{v> N \log^{\nicefrac{-2}{\gamma}}N} \de(v)=O_{\mathbb{P}}^{\log N}(1)
\]
to say that the conductance of any edge $\{v,w\}$ incident to a vertex $v \notin S$ has
\[
c(v,w)
\geq
\frac{1}{\left| \sC_{\rm max} \right|} \frac{\de(v)^\theta}{2}
\geq
\frac{1}{N} \frac{\de(v)^\theta}{2}
\]
and so uniformly bound the expected hitting of each pair on the order $O_{\mathbb{P}}^{\log N}(N)$ via the electrical network bound of Proposition \ref{prop:max_resistance}. This induces a bound on the meeting time by \cite[Proposition 14.5]{aldous-fill-2014} and then the required $O_{\mathbb{P}}^{\log N}(N)$ bound on the expected consensus time by Proposition \ref{prop:coal}.

Further we have an upper bound on the small components, which by \cite[Theorem 3.12(ii)]{bollobas2007phase} have maximal size $O_{\mathbb{P}}(\log N)$, by the same hitting time argument -- using that the maximum degree among these components is bounded by the maximum component size.

The lower bound follows by an application of Proposition \ref{lower_meeting_bound} to $\sC_{\rm max}$. This simplifies, because the stationary distribution is uniform on $\sC_{\rm max}$, to $\Omega_{\mathbb{P}}(N^2/\sum_v q(v))$. Given $\theta \leq {\nicefrac{1}{\gamma}}$ we find (recalling also the degree approximation \eqref{eq_degree_handwaving})
\[
\sum_{v \in \sC_{\rm max}}q(v)
=
\sum_{v \in \sC_{\rm max}}\de(v)^\theta
=
O_{\mathbb{P}}^{\log N}\left(
\sum_{v=1}^N
\left(
\frac{N}{v}
\right)^{\gamma \theta}
\right)
=
O_{\mathbb{P}}^{\log N}(N)
\]
and so we have a lower bound which is polylogarithmically tight to the upper bound
\[
(4+o(1)) \, 
t_{\text{\textnormal{meet}}}^{\pi} \geq
\frac{N^2}{\sum_v q(v)}
=\Omega_{\mathbb{P}}^{\log N}(N).
\]


{\bf Case 2:} Next we address $\frac{1}{\gamma}<\theta \leq 2$.

In this case the lower bound of Proposition \ref{lower_meeting_bound} uses
\[
\sum_{v \in \sC_{\rm max}}q(v)
=
\sum_{v \in \sC_{\rm max}}\de(v)^\theta
=
O_{\mathbb{P}}^{\log N}\left(
\sum_{v=1}^N
\left(
\frac{N}{v}
\right)^{\gamma \theta}
\right)
=
O_{\mathbb{P}}^{\log N}\left(N^{\gamma \theta}\right)
\]
and so we instead find the lower bound  $t_{\text{\textnormal{meet}}}^{\pi} 
=\Omega^{\log N}_{\mathbb{P}}\left(N^{2-\gamma \theta}\right).$

For the upper bound, we will need a partially observed version of the chain (as defined in Definition \ref{define_partial_obs}) using the set of leaf neighbours of the vertex $1 \in \sC_{\rm max}$
\[
L_1:=
\left\{
v \in [N]:
v \sim 1,
\de(v)=1
\right\}.
\]

Let $(X^{(1)}_t)_t$, $(X^{(2)}_t)_t$ denote two i.i.d. Markov chains with the dual dynamics.
Then for some constant $C>0$ define a timeframe 
\[
T=C N^{2-\gamma\theta}t_{\rm rel}
\]
and we control the occupation time of the two independent walkers in $L_1 \times L_1$:
\[
\sigma_{L_1 \times L_1}(t):=\int_0^t \mathbbm{1}_{(X^{(1)}_s,X^{(2)}_s) \in L_1 \times L_1} {\rm d}s.
\]

Write
\[
\mathbb{P}_{\pi \otimes \pi}\left(
\sigma_{L_1 \times L_1}(T)
<
\frac{T \pi\left( L_1 \right)^2}{2}
\right)
=\mathbb{P}_{\pi \otimes \pi}\left(
\frac{1}{T}
\int_0^T\left(
\pi\left( L_1 \right)^2-\mathbbm{1}_{L_1 \times L_1} \left( X^{(1)}_t,X^{(2)}_t \right)
\right){\rm d}t
>
\frac{\pi\left( L_1 \right)^2}{2}
\right).
\]
and then apply \cite[Remark 1.2]{lezaud_chernoff} with
\[
f:=\pi\left( L_1 \right)^2-\mathbbm{1}_{L_1 \times L_1}
\]
(and hence 
$
\left\|
f
\right\|^2_{2,\pi}
\sim
\pi\left( L_1 \right)^2
$) 
to find
\begin{equation}\label{eq_chernoff}
\mathbb{P}_{\pi \otimes \pi}\left(
\frac{1}{T}
\int_0^T f
> \frac{\pi\left( L_1 \right)^2}{2}
\right)
\leq
\exp
\left(
-(1+o(1))\frac{T \pi\left( L_1 \right)^2}{32 t_{\rm rel}}
\right).
\end{equation}

Note also by exploring the network from vertex $1$ and considering only edges into the smaller half of the vertices by weight, that
\[
\left|
L_1
\right|
\succeq
\operatorname{Bin}
\left(
\left\lfloor \nicefrac{N}{2} \right\rfloor,
(1+o(1)) \, 
\beta N^{\gamma-1}
e^{-\frac{\beta 2^\gamma}{(1-\gamma)}}
\right)
\]
and so, given $\beta>1$, we can guarantee $\pi\left( L_1 \right) \geq
\tfrac{1}{3}  N^{\gamma-1}
e^{-\frac{\beta 2^\gamma}{(1-\gamma)}}$ with high probability. 
Using also that $\theta \leq 2$ we can bound \eqref{eq_chernoff} with high probability by
\[
\exp
\left(
-\frac{T \pi\left( L_1 \right)^2}{33 t_{\rm rel}}
\right)
\leq
\exp
\left(
-\frac{C e^{-\frac{\beta 2^{1+\gamma}}{(1-\gamma)}}}{99 }
N^{2-\gamma\theta} N^{2\gamma-2}
\right)
\leq
\frac{1}{3}
\]
for some large constant $C$ and large enough $N$.

We then consider the partially observed coalescence dynamics which are conveniently simple. If $\pi$ denotes as usual the uniform measure on the giant $\sC_{\rm max}$, then
\[
\mathbb{P}_{\pi \otimes \pi}\left(
\sigma_{L_1 \times L_1}(\tau_{\rm meet})>t
\right)
=
\left(
1-\frac{1}{|L_1|}
\right)
e^{-\tfrac{2t}{|L_1|}\left( \tfrac{\de(1)^{\theta-1}+1}{2} \right)}
\]
because they are initially coincident with probability $\nicefrac{1}{|L_1|}$, and from then $\tfrac{\de(1)^{\theta-1}+1}{2}$ is the constant rate of meeting for the partially observed walkers. In particular, this means
\[
\mathbb{P}_{\pi \otimes \pi}\left(
\sigma_{L_1 \times L_1}(\tau_{\rm meet})>T
\right)
=e^{-\Theta_{\mathbb{P}}^{\log N}\left(
N^{2-2\gamma}
\right)}
=o_{\mathbb{P}}(1)
\]
and so have the required occupancy event of \eqref{eq_chernoff} and moreover observe a meeting before time $T$, with probability at least $\nicefrac{1}{2}$. Hence, by restarting after failure to meet, we expect to see meeting before time $2T$. 
We conclude
\[
2T=  O_{\mathbb{P}}^{\log N}\left(
N^{2-\gamma \theta}
\right)
\]
by the bound of Theorem \ref{theorem_erw_mixing} on the relaxation time which applies to every $\theta >\nicefrac{1}{\gamma}$ by Proposition \ref{prop_mix_monotone}, and finally Proposition \ref{prop:coal} shows that coalescence time on the network is logarithmically comparable to slowest expected meeting.
\end{proof}

{\bf Acknowledgements.}
JF was supported by a scholarship from the EPSRC Centre for Doctoral Training in Statistical Applied Mathematics at Bath (SAMBa), under the project number EP/L015684/1, then by the \emph{Unit\'e de math\'ematiques pures et appliqu\'es} of ENS Lyon, and now by NKFI grant KKP 137490.

\printbibliography

\begin{appendices}
\section{Appendix}

The following integral condition is directly what we need for subcriticality of the pendant trees in the construction of the proof of Theorem \ref{theorem_erw_mixing}, but we want to find simpler parameter regions for the resultant main theorems.

\begin{proposition}\label{prop_beta_3}
For the unique $\rho \in (0,1)$ satisfying $1-\rho=e^{-\beta \rho}$, the integral condition
\[
I(\beta):=\int_0^1
(1-\rho)^{\left( \frac{1}{1-\gamma}-1+\gamma \right) x^{-\gamma}}
x^{-2\gamma}
{\rm d} x<1
\]
is guaranteed by $\beta\geq 3$. 
\end{proposition}

\begin{proof}
First we change variables $u=x^{-\gamma}$ and observe that this integral is a product of two decreasing functions of $u$
\[
\begin{split}
I(\beta)
&=
\int_1^\infty
\frac{u^{1-\frac{1}{\gamma}}}{\gamma}
(1-\rho)^{\left( \frac{1}{1-\gamma}-1+\gamma \right) u}
{\rm d} u\\
&\leq
\int_1^\infty
\frac{1}{\gamma}
(1-\rho)^{\left( \frac{1}{1-\gamma}-1+\gamma \right) u}
{\rm d} u
\wedge
\int_1^\infty
\frac{u^{1-\frac{1}{\gamma}}}{\gamma}
(1-\rho)^{\left( \frac{1}{1-\gamma}-1+\gamma \right)}
{\rm d} u.
\end{split}
\]

These are then straightforward to calculate:
\[
\int_1^\infty
\frac{(1-\rho)^{\left( \frac{1}{1-\gamma}-1+\gamma \right) u}}{\gamma}
{\rm d} u
=
\frac{-(1-\rho)^{ \frac{1}{1-\gamma}-1+\gamma}}{\gamma\left( \frac{1}{1-\gamma}-1+\gamma \right) \log (1-\rho) }
<
\frac{1}{\gamma\left( \frac{1}{1-\gamma}-1+\gamma \right) \log \left( \frac{1}{1-\rho} \right) };
\]
and if $\gamma<\frac{1}{2}$
\[
\int_1^\infty
\frac{u^{1-\frac{1}{\gamma}}}{\gamma}
(1-\rho)^{\left( \frac{1}{1-\gamma}-1+\gamma \right)}
{\rm d} u
=
\frac{(1-\rho)^{\left( \frac{1}{1-\gamma}-1+\gamma \right)}}{1-2\gamma}.
\]

Rearranging, we have $I(\beta)<1$ if either
\begin{equation}\label{eq_I_ineq}
\log \left( \frac{1}{1-\rho} \right)>\frac{1}{\gamma\left( \frac{1}{1-\gamma}-1+\gamma \right) }
\end{equation}
or $\gamma<\frac{1}{2}$ and
\[
\log \left( \frac{1}{1-\rho} \right)>
\frac{-\log \left(1-2\gamma\right)}{\frac{1}{1-\gamma}-1+\gamma}.
\]

We check that the first bound is strictly decreasing in $\gamma$ and the second is strictly increasing, and so we can set $\gamma_0=0.45$ and calculate
\[
\frac{1}{\gamma_0\left( \frac{1}{1-\gamma_0}-1+\gamma_0 \right)}
\wedge
\frac{-\log \left(1-2\gamma_0\right)}{\frac{1}{1-\gamma_0}-1+\gamma_0}
\approx
1.752 \dots 
\wedge
1.816 \dots 
\]

Rearranging again, we conclude that this is satisfied when
\[
\rho>1-e^{-1.816 \dots}
\]
i.e.
\[
\beta=\frac{1}{\rho}\log \frac{1}{1-\rho}>
\frac{1.816\dots}{1-e^{-1.816\dots}}
=2.169\dots
\]
\end{proof}

\begin{proposition}\label{prop_beta_2log}
For the unique $\rho \in (0,1)$ satisfying $1-\rho=e^{-\beta \rho}$, the integral condition
\[
I(\beta):=\int_0^1
(1-\rho)^{\left( \frac{1}{1-\gamma}-1+\gamma \right) x^{-\gamma}}
x^{-2\gamma}
{\rm d} x<1
\]
is guaranteed, if $\gamma \geq \nicefrac{1}{2}$, by $\beta> 2\log 2$.
\end{proposition}

\begin{proof}
Given the assumption $\rho>\nicefrac{1}{2}$ which is equivalent to $\beta> 2\log 2$, the previous integral $I(\beta)$ is bounded by
\[
\int_0^1
f(\gamma,x)
{\rm d} x,
\qquad \text{where} \qquad
f(\gamma,x):=
2^{-\left( \frac{1}{1-\gamma}-1+\gamma \right) x^{-\gamma}}
x^{-2\gamma}.
\]

We show that this integrand $f$ is decreasing in $\gamma$ at all relevant $x,\gamma$
\begin{equation}\label{eq_leapoffaith}
\begin{split}
(1-\gamma)^2 2^{\frac{(2-\gamma) \gamma x^{-\gamma}}{1-\gamma}} x^{3 \gamma}
\frac{\partial f}{\partial \gamma}(\gamma,x)&=2 (1-\gamma)^2 x^\gamma \log \frac{1}{x}\\
&+  \gamma \left(3\gamma -\gamma^2 -2\right)\log (2)\log \frac{1}{x}\\
&+\left(2\gamma-\gamma^2  -2\right)\log (2).
\end{split}
\end{equation}

Then by finding stationary points of the pieces:
\[
\sup_{\gamma \in (\nicefrac{1}{2},1)}
\sup_{x \in (0,1)} 2 (1-\gamma)^2 x^\gamma \log \frac{1}{x}
=
\sup_{\gamma \in (\nicefrac{1}{2},1)}\frac{2(1-\gamma)^2}{\gamma e}
=\frac{1}{e};
\]
\[
\sup_{\gamma \in (\nicefrac{1}{2},1)}\left(3\gamma -\gamma^2 -2\right)=0;
\]
\[
\sup_{\gamma \in (\nicefrac{1}{2},1)}\left(2\gamma -\gamma^2 -2\right)=-1.
\]

Thus we bound the right hand side of \eqref{eq_leapoffaith} by $\frac{1}{e}-\log 2<0$, and so the integral also is strictly decreasing in $\gamma$. We can bound it by what we recognise as an incomplete gamma function $\Gamma$
\[
\int_0^1
f(\gamma,x)
{\rm d} x
<
\int_0^1
f\left(\frac{1}{2},x\right)
{\rm d} x
=
\int_0^1
\frac{2^{-\nicefrac{3}{2\sqrt{x}}}}{x}
{\rm d} x=2\Gamma\left(0,\frac{3 \log[2]}{2}\right)\approx 0.41
\]
in particular this integral is less than $1$.
\end{proof}

The giant component regime on Norros-Reittu networks is widely accepted folklore, and a proof is suggested in \cite[Remark 2.4]{bollobas2007phase}, but for unambiguity and for illustration we will adapt one from the results on Chung-Lu type networks in that work.

\begin{proposition}\label{prop_rankone_snr_giant}
The Norros-Reittu models of Defintions \ref{def_mnr} and \ref{def_gsnr} have a giant component if
\[
\int_0^1 \operatorname{f}^2(x) {\rm d}x\in (1,\infty ] .
\]
\end{proposition}

\begin{proof}
If $\int \operatorname{f}^2=\infty$ then take $\epsilon \in (0,1)$ such that $\epsilon=e^{-\epsilon}$. 
We have the bound for any $x>0$
\[
1-e^{-x}>x(1-x)
\]
so that for $i,j \in [N]$ such that $w_{ij}:=\tfrac{1}{N}\operatorname{f}\left(\tfrac{i}{N},\tfrac{j}{N}\right)<\epsilon$ we can infer for the SNR edge probabilities
\[
p_{ij}=1-\exp \left( -w_{ij} \right)
>(1-\epsilon) w_{ij}
.
\]
Otherwise, if $w_{ij}\geq\epsilon$, we have instead
\[
p_{ij}=1-\exp \left( -w_{ij}\right)
 \geq 1-\exp \left( -\epsilon \right)
 \geq
 (1-\epsilon)\left(1 \wedge w_{ij} \right)
\]
where the last inequality follows from recalling $\epsilon=e^{-\epsilon}$. Combining both bounds, we have shown that the SNR model dominates the Chung-Lu model after edge percolation of the latter with retention probability $1-\epsilon$.
For this percolated Chung-Lu model we can apply \cite[Corollary 3.3]{bollobas2007phase}, noting that $
\left\| T_\kappa \right\|=
(1-\epsilon)^2\int_0^1 \operatorname{f}^{2}
=\infty
$
and so we find a giant component in the subgraph with high probability.

If instead $1<\int \operatorname{f}^2<\infty$ we apply \cite[(6.8.13)]{van2016random} to say that this graph contains the GRG version with edge probabilities
\[
p_{ij}=\frac{\operatorname{f}\left( \tfrac{i}{N} \right) \operatorname{f}\left( \tfrac{j}{N} \right)}{N+\operatorname{f}\left( \tfrac{i}{N} \right) \operatorname{f}\left( \tfrac{j}{N} \right)}.
\]

By \cite[Theorem 6.10]{van2016random} the empirical degree distribution of this graph converges in $L^1$ to the mixed Poisson law $D\sim\operatorname{Pois}(W)$ where
\[
W
\stackrel{({\rm d})}{=}
\operatorname{f}\left(
U_{[0,1]}
\right)
\int_0^1 \operatorname{f}
\]
and $U_{[0,1]}$ denotes a random variable uniformly distributed on $[0,1]$. By \cite[Theorem 6.15]{van2016random} this graph is exactly uniformly distributed, conditional on its degrees, among simple graphs with these degrees.

Then \cite[Theorem 1]{bollobas2015old} tells us that the uniform simple graph has a giant component when the \emph{size-biased} limit law of its degree distribution has mean larger than $2$. Recalling the $\operatorname{Pois}(\lambda)$ second moment is $\lambda+\lambda^2$, this is when
\[
2<
\frac{\mathbb{E}\left( D^2\right)}{\mathbb{E}\left( D\right)}
=
\frac{\mathbb{E}\left( \operatorname{f} \left( U_{[0,1]} \right)\right)+\mathbb{E}\left( \operatorname{f} \left( U_{[0,1]} \right)^2\right)\int_0^1 \operatorname{f}}{\mathbb{E}\left( \operatorname{f} \left( U_{[0,1]} \right)\right)}=1+\int_0^1 \operatorname{f}^2
\]
which is the given condition. Then, because we have found a subgraph with a giant component, we conclude that the SNR model must also have a giant.
\end{proof}

The idea behind the diameter bound of Theorem \ref{my_diameter_bound} is that if we have an independent positive probability to mark evey vertex in a graph and then identify every marked vertex, the resultant multigraph has componentwise diameter $O_{\mathbb{P}}(\log N)$. We prove this in the following lemma.

\begin{lemma}\label{percolation_lemma}
Fix $p \in (0,1)$. On a sequence of graphs on $[N]$, mark each vertex with independent probability $p$. Then for any vertex define $D(v)$ as the minimum distance to a marked vertex. We have
\[
\max_{v} \left( D(v) \wedge \operatorname{diam}\mathscr{C}(v) \right) =O_{\mathbb{P}}(\log N).
\]
\end{lemma}

\begin{proof}
We now construct the sets in which to observe arrivals. Any vertex $v$ has a path from it of length at least $\operatorname{diam}\mathscr{C}(v)/2$, and so if we set
\[
V_x:=\{v \in [N] : \operatorname{diam}\mathscr{C}(v) \geq 2x \}
\]
we can give each vertex $v \in V_x$ a simple path $P(v)$ with $|P(v)|=x$ and $v$ at one of the ends of $P(v)$. Define $E(v)$ to be the event that $P(v)$ contains a marked vertex, and then
\[
\mathbb{P}(D(v)< x)\geq \mathbb{P}\left(E(v)\right) = 1- (1-p)^{x}
\]
so that if we then set $x=\left\lceil \frac{2\log N}{-\log(1-p)}\right\rceil=\Theta(\log N)$ and apply Harris' inequality \cite{harris1960lower}

\[
\begin{split}
\mathbb{P} \left( \forall v \in V_x : \ D(v)< x \right) &\geq \mathbb{P}\left(\bigcap_{v \in V_x}E(v)\right) \geq \prod_{v \in V_x} \mathbb{P}\left(E(v)\right)\\
&= \left( 1-(1-p)^{x} \right)^N \geq \left( 1-\frac{1}{N^2} \right)^N \rightarrow 1.\\
\end{split}
\]
\end{proof}

For the purposes of proving results with polylogarithmic corrections, it is frequently useful to find a large set of vertices which can be uniformly treated as of degree $O_{\mathbb{P}}^{\log N}(1)$.

\begin{lemma}\label{lemma_log_degrees}
For the SNR network $G_N$ with any parameters $\beta>0$, $\gamma \in (0,1)$ we find, for any $\epsilon>0$ and $\alpha \geq \nicefrac{1}{\gamma}$,
\[
\max_{v \geq \epsilon N \log^{-\alpha} N} \de(v) =
O_{\mathbb{P}}\left( \log^{\alpha\gamma} N \right).
\]
\end{lemma}
\begin{proof}
The Norros-Reittu network of Definition \ref{def_SNR} is a simplified version of the multigraph version of Definition \ref{def_mnr} with $\operatorname{f}(x)=\sqrt{\beta}x^{-\gamma}$. This model has degree distributions $\de(v) \sim \operatorname{Pois}\left( w(v) \right)$, where
\[
w(v)
=
\frac{1}{N}
\operatorname{f}\left(
\frac{v}{N}
\right)
\sum_{x=1}^N \operatorname{f}\left(
\frac{x}{N}
\right)
<
\beta N^{2\gamma-1} v^{-\gamma}
\int_0^N x^{-\gamma} {\rm d} x
=
\frac{\beta}{1-\gamma}\left(\frac{N}{v}\right)^\gamma
\]
so that
\[
\max_{v \geq \epsilon N \log^{-\alpha} N} w(v)
<
\frac{\beta \epsilon^{-\gamma}}{1-\gamma} \log^{\alpha\gamma} N.
\]

Recall also that in the MNR version we have exactly degrees with Poisson distribution according to their MNR weights. Hence, because the SNR version is constructed from the MNR version by flattening, we have the stochastic order
\[
\de(v) \preceq \operatorname{Pois}\left(\frac{\beta \epsilon^{-\gamma}}{1-\gamma} \log^{\alpha\gamma} N\right).
\]

If we fix a constant $C$ with $e^C-1=\frac{1-\gamma}{\beta \epsilon^{-\gamma}}$ then we conclude
\[
\mathbb{E}\left(e^{C \de(v)}\right) < e^{\log^{\alpha\gamma} N},
\]
and so by the Chernoff bound we deduce that $C\de(v)< 3 \log^{\alpha\gamma} N$ with probability $1-o\left(\nicefrac{1}{N}\right)$. The result follows by the union bound.
\end{proof}

Large degrees $\de(v)$ in the network are well approximated by $\left(\frac{N}{v}\right)^\gamma$, and so in the following lemma we check that a large proportion of edges from every high degree vertex are pointing into the Erd\H{o}s-R\'enyi giant.

\begin{lemma}\label{lemma_high_internal_edges}
Let $\Gamma(v)$ denote the neighbourhood 
$
\Gamma(v)=\left\{w \in [N] : w \sim v \right\}.
$ 
Then if we take $\epsilon$ small enough such that $\beta \rho \epsilon^{-\gamma}>32$, in the construction of the proof of Theorem \ref{theorem_erw_mixing} we have
\[
\min_{v \in H_\epsilon} \frac{\left|\Gamma(v)\cap \sC_1\right|}{\left(\frac{N}{v}\right)^\gamma}
=
\Omega_{\mathbb{P}}\left(\frac{1}{\log N} \right).
\]
\end{lemma}

\begin{proof}

Between any  $v\leq\epsilon N \log^{-\frac{1}{\gamma}} N$ and any $w \in \sC_1$ we expect in the MNR a number of edges at least
\[
\frac{\beta}{N}
\left(\frac{v}{N}\right)^{-\gamma}.
\]

Working on the high probability assumption that $\left| \sC_1 \right| \geq \tfrac{\rho}{2} N$ and using that the SNR graph dominates the GRG graph (i.e. that $1-e^{-p}\geq \frac{p}{1+p}$) we claim
\[
\left|\Gamma(v)\cap \sC(1)\right|
\succeq
\operatorname{Bin}
\left(
\frac{\rho}{2} N,
\frac{\frac{\beta}{N}
\left(\frac{v}{N}\right)^{-\gamma}
}{1+\beta N^{\gamma-1}}
\right)
\succeq
\operatorname{Bin}
\left(
\frac{\rho}{2} N,
\frac{\frac{\beta}{N}
\left(\frac{v}{N}\right)^{-\gamma}
}{2}
\right).
\]

Hence by the usual multiplicative Chernoff bound these edges number less than $\frac{\beta \rho }{8} \left(\frac{N}{v}\right)^\gamma$ with probability asymptotically bounded by
\[
\exp \left(
-\frac{1}{8} \cdot \frac{\beta \rho}{4} \left(\frac{N}{v}\right)^\gamma
\right)
\leq
\exp \left(
-\frac{\beta \rho}{32} \epsilon^{-\gamma} \log N
\right)
=o\left(\frac{1}{N}\right),
\]
by the assumption $\beta \rho \epsilon^{-\gamma}>32$. The conclusion follows from the union bound over every vertex in $H_\epsilon$.
\end{proof}

Finally, we have a fairly simple adaptation of \cite[Proposition 14.11]{aldous-fill-2014} to 
reducible Markov chains.

\begin{proposition}\label{prop:coal} On a disconnected graph with components $\{C_j : j \in [k]\}$ we find
	\[ \sup_{j \in [k]} t_{\rm meet}(C_j) \leq t_{\rm coal} \leq e(2+ \log N ) \sup_{j \in [k]} t_{\rm meet}(C_j) .\]
\end{proposition}

\begin{proof}
The reversible Markov chain decomposes into irreducible recurrence classes - write $\mathscr{C}(i)$ for the class containing the state $i$. As in the proof of \cite[Proposition 14.11]{aldous-fill-2014}, consider a walker $W^{(i)}$ independently started in $i$. We have $\nicefrac{N(N+1)}{2}$ meeting times
\begin{equation}\label{eq:meeting_ij}
\tau^{i,j}_{\text{meet}}:=\inf
\left\{
t \geq 0 :
W^{(i)}_t=W^{(j)}_t
\right\}
\end{equation}
for the walkers $1 \leq i \leq j \leq N$, where we define $\inf \emptyset := \infty$ and $\tau^{i,i} := 0$. Define a function $\operatorname{f}$ which maps all elements in a recurrence class $\mathscr{C}(i)$ to a label $\min \mathscr{C}(i)$ which is of lowest index in that component
\[
\operatorname{f}: i \mapsto \min \mathscr{C}(i).
\]

Then we can construct the coalescing walker from independent walkers by killing the walker of larger initial position at any meeting event, which we think of as making it follow the vertex of smaller initial position. Thus we can say, for the non-independent walker meeting times obtained in this construction,
\[
\tau_{\text{coal}}:=\max_{i=1}^N \tau_{\text{coal}}(\mathscr{C}(i))\leq \max_{i=1}^N \tau^{i,f(i)}_{\text{meet}}.
\]

We then apply a result for the general exponential tails of hitting times of finite Markov chains \cite[Equation 2.20]{aldous-fill-2014}: from arbitrary initial distribution $\mu$ and for a continuous time reversible chain, for any subset $A\subset V$
\[
\mathbb{P}_{\mu}(T_A>t)\leq \exp \left( - \left\lfloor \frac{t}{e \max_v \mathbb{E}_v T_A } \right\rfloor  \right).
\]

For the meeting time variables, which are hitting times for the product chain, this leads to
\[
\mathbb{P}(\tau^{i,j}_{\text{meet}}>t)\leq \exp\left( - \left\lfloor \frac{t}{e t_{\text{meet}}} \right\rfloor \right) . 
\]

We can deduce by the union bound that
\[\mathbb{P}(\tau_{\text{coal}}>t)
\leq \sum_{i=1}^N \mathbb{P}\left(\tau^{f(i),i}_{\text{meet}}>t\right)
\leq N\exp\left( - \left\lfloor \frac{t}{e t_{\text{meet}}} \right\rfloor \right).
\]

Finally, we integrate as in \cite[Proposition 14.11]{aldous-fill-2014} to get
\[
t_{\rm coal}\leq
\int_0^\infty
1 \wedge
\left(
N e \exp\left( - \frac{t}{e t_{\text{meet}}} \right)
\right) {\rm d} t
=
e \left(
2+\log N
\right) t_{\text{meet}}, 
\]
which proves the upper bound.
\end{proof}

\end{appendices}

\end{document}